\documentclass[11pt,reqno,oneside]{amsart}
\usepackage[utf8]{inputenc}
\usepackage[english]{babel}
\usepackage[babel]{csquotes}
\usepackage[backend=bibtex, hyperref]{biblatex}
\bibliography{BandCor}

\usepackage{amsmath}
\usepackage{amsfonts}
\usepackage{amssymb}
\usepackage{amsthm} 
\usepackage{dsfont} 
\usepackage[mathscr]{euscript} 
\usepackage{eurosym}
\usepackage{bbm} 
\usepackage[a4paper]{geometry}
\usepackage{comment} 
\usepackage{paralist} 
\usepackage{units} 

\usepackage{graphicx}
\usepackage{xkeyval}
\usepackage{pstricks}
\usepackage{hyperref}
\usepackage{longtable}
\usepackage{tikz}

\usepackage{apptools} 
\AtAppendix{\counterwithin{lemma}{section}}

\hypersetup{
urlcolor=black, 
  menucolor=black, 
  citecolor=black, 
  anchorcolor=black, 
  filecolor=black, 
  linkcolor=black, 
  colorlinks=true,
}
\usepackage{comment}

\newcommand{\lebesgue}{\lambda\mspace{-7mu}\lambda}

\newcommand{\Cat}{\mathscr{C}}
\newcommand{\Prob}{\mathds{P}}

\newcommand{\E}{\mathds{E}}
\newcommand{\V}{\mathds{V}}
\newcommand{\C}{\mathbb{C}}
\newcommand{\R}{\mathbb{R}}

\newcommand{\N}{\mathbb{N}}

\newcommand{\abs}[1]{|{#1}|} 
\newcommand{\bigabs}[1]{\left|{#1}\right|} 
\newcommand{\one}{\mathds{1}}

\newcommand{\Acal}{\mathcal{A}}

\newcommand{\Bcal}{\mathcal{B}}

\newcommand{\Pcal}{\mathcal{P}}

\newcommand{\Ncal}{\mathcal{N}}
\newcommand{\Ecal}{\mathcal{E}}

\newcommand{\Tcal}{\mathcal{T}}

\newcommand*{\defeq}{\mathrel{\vcenter{\baselineskip0.5ex \lineskiplimit0pt
                     \hbox{\scriptsize.}\hbox{\scriptsize.}}}%
                     =}
										
\newcommand{\integrala}[2]{\left\langle{#1},{#2}\right\rangle}

\newcommand{\de}{\text{d}}

\newcommand{\ubar}[1]{\underline{#1}}
\newcommand{\oneto}[1]{[{#1}]}

\DeclareMathOperator{\Cov}{Cov}

\renewcommand{\Im}{\operatorname{Im}}

\DeclareMathOperator{\tr}{tr}

\DeclareMathOperator{\dist}{dist}

\theoremstyle{plain}
\newtheorem{lemma}{Lemma}
\newtheorem{theorem}[lemma]{Theorem}

\newtheorem{corollary}[lemma]{Corollary}
\theoremstyle{definition}
\newtheorem{definition}[lemma]{Definition}

\newtheorem{remark}[lemma]{Remark}
\theoremstyle{remark}

\begin{document}
\title[Random Band and Block Matrices with Correlated Entries]
{Random Band and Block Matrices with Correlated Entries}
\author[Riccardo Catalano]{Riccardo Catalano}
\author[Michael Fleermann]{Michael Fleermann}
\author[Werner Kirsch]{Werner Kirsch}
\begin{abstract}
In this paper, we derive limit laws for the empirical spectral distributions of random band and block matrices with correlated entries. In the first part of the paper, we study band matrices with \emph{approximately uncorrelated} entries. We strengthen previously obtained results while requiring weaker assumptions, which is made possible by a refined application of the method of moments. In the second part of the paper, we introduce a new two-layered correlation structure we call \emph{SSB-HKW correlated}, which enables the study of structured random matrices with correlated entries. Our results include semicircle laws in probability and almost surely, but we also obtain other limiting spectral distributions depending on the conditions. Simple necessary and sufficient conditions for the limit law to be the semicircle are provided. Our findings strengthen and extend many results already known.
\end{abstract}
\keywords{random band matrices, block matrices, correlated entries, semicircle law}
\subjclass[2020]{60B20} 
\maketitle

\section{Introduction}

This paper is concerned with the study of random band and block matrices with correlated entries. We extend the works of \cite{Hochstattler:Kirsch:Warzel:2016}, \cite{Fleermann:Kirsch:Kriecherbauer:2021}, and \cite{Schenker:Schulz-Baldes:2005} into various directions, and we include and extend the so far unpublished results of the dissertation \cite{Catalano:2016} of the first author. The authors of \cite{Schenker:Schulz-Baldes:2005} have shown the semicircle law in expectation for an ensemble of random matrices that exhibit dependencies within certain equivalence classes of matrix entries, while independence is assumed between different equivalence classes.The authors of \cite{Hochstattler:Kirsch:Warzel:2016} have developed an ensemble they called \emph{approximately uncorrelated}, allowing for equally strong dependencies between all matrix entries irrelevant of their distance, and the semicircle law in probability was observed. In \cite{Fleermann:Kirsch:Kriecherbauer:2021}, the setup of random band matrices with approximately uncorrelated entries has been analyzed deeper, deriving conditions for the semicircle law to hold in probability and almost surely, extending the work of \cite{Hochstattler:Kirsch:Warzel:2016}. For some of the results, the notion of approximately uncorrelated ensembles was tightened to ensembles called \emph{$\alpha$-almost uncorrelated}.

By now, the literature on random matrices with correlated entries has grown tremendously, while in the case of structured random matrices such as band and block matrices, results remain sparse for the correlated case. Therefore, one of our aims is to contribute to this domain. We give a short summary of relevant previous studies: After the seminal paper \cite{Schenker:Schulz-Baldes:2005}, random matrices with more specific correlation structures have been studied in \cite{Friesen:Lowe:2013:a}, \cite{Friesen:Lowe:2013:b}, and \cite{Lowe:2019}, all of which consider the case of symmetric random matrices with independent diagonals, but with different assumption on the correlations within these diagonals. Further, the authors of \cite{Banna:Merlevede:Peligrad:2015} study the case of symmetric random matrices with entries which are functions of shifted i.i.d.\ random fields. The papers \cite{Hochstattler:Kirsch:Warzel:2016}, \cite{Kirsch:Kriecherbauer:2018}, and \cite{Fleermann:Kirsch:Kriecherbauer:2021} focus on semicircle laws for approximately uncorrelated entries and in particular Curie-Weiss models, with a local version of their results contained in \cite{Fleermann:Kirsch:Kriecherbauer:2021:b}. Further results in the local regime were achieved in \cite{Che:2017} for finite range correlations, in \cite{Ajanki:Erdos:Kruger:2016} for correlated Gaussian entries and in \cite{Erdos:Kruger:2017} for more general correlations.
For structured random matrices with correlations, see
\cite{Kemp:Zimmermann:2020} (and references therein) for arbitrary correlations within $\log$-size families of matrix entries, but with independence between these families, \cite{Kirsch:Kriecherbauer:2020} for random band matrices with exchangeable entries and \cite{Fleermann:Kirsch:Kriecherbauer:2021} for random band matrices with approximately uncorrelated entries.

This paper is organized as follows: Section~\ref{sec:bandmatrix} contains the first main theorem of this paper, Theorem~\ref{thm:bandmatrix}. There, we show that the additional assumptions made in \cite{Fleermann:Kirsch:Kriecherbauer:2021} can be dropped  without replacement (thus only imposing the conditions as in \cite{Hochstattler:Kirsch:Warzel:2016}), while obtaining stronger results. These improvements are achieved by employing a refined way of applying the method of moments, where we use observations in \cite{Fleermann:2015} (which were also presented in \cite{Fleermann:2019}). Further, by a standard perturbation argument, we also derive that our results remain true for non-periodic band matrices. In Section~\ref{sec:weightedmatrix} we formulate the second main result of this paper: Theorem~\ref{thm:weightedmatrix}. This theorem was initially motivated by an extension of the results in \cite{Fleermann:Kirsch:Kriecherbauer:2021} to band matrices with correlated entries and a bandwidth of proportional growth. However, we study this problem in a much more general setting, including and extending the yet unpublished results \cite{Catalano:2016}, also generalizing the findings in \cite{Molchanov:Pastur:Khorunzhy:1992}. First, we revisit the work \cite{Schenker:Schulz-Baldes:2005} and drop the assumption that entries from different equivalence classes should be independent. Instead, matrix entries are assumed to be approximately uncorrelated between different equivalence classes, and arbitrarily correlated within these classes. Thus, our model has two layers of correlations and will be called \emph{SSB-HKW correlated ensemble}, in honor of the authors that invented these correlation structures. Second, as in \cite{Molchanov:Pastur:Khorunzhy:1992}, we introduce a weight function which will then allow us to study band and block matrices, and also allow for limits different from the semicircle distribution. Necessary and sufficient conditions for the semicircle law are provided and coincide with the conditions in \cite{Molchanov:Pastur:Khorunzhy:1992}. Lastly, while the findings in \cite{Schenker:Schulz-Baldes:2005} pertained to weak convergence in expectation, we show that their results also hold weakly in probability, and we derive mild sufficient conditions for the convergence results to also hold weakly almost surely. Examples that fit into the framework of our main Theorems~\ref{thm:bandmatrix} and~\ref{thm:weightedmatrix} include Curie-Weiss and correlated Gaussian ensembles.

\section{Random band matrices with approximately uncorrelated entries}
\label{sec:bandmatrix}
\subsection{Setup and Results}
\label{sec:AUsetup} Let $(a_n)_n$ be a sequence of random matrices, where $a_n$ is a symmetric $n\times n$ matrix and the random variables $a_n(p,q)$ are real-valued. Using the notation $\oneto{\ell}\defeq\{1,\ldots,
\ell\}$ for $\ell\in\N$, we will call pairs $(p_1,q_1),\ldots,(p_{\ell},q_{\ell})\in\oneto{n}^2$ fundamentally different, if for all $1\leq i\neq j \leq \ell: \{p_i,q_i\}\neq\{p_j,q_j\}$. Note that if $P_1,\ldots P_{\ell}\in\oneto{n}^2$ are fundamentally different, then $a_n(P_1),\ldots,a_n(P_{\ell})$ are pairwise different matrix elements, even modulo symmetry.
 We assume the following conditions:
There exist sequences of constants $(C(k))_k$ and $(C^{(\ell)}_n)_{n,\ell}$ such that for all $\ell\in\N$, $\lim_{n\to\infty} C^{(\ell)}_n = 0$ and such that for all $N,\ell\in\N$ and all fundamentally different pairs $P_1,\ldots,P_{\ell}\in\oneto{n}^2$,
\begin{align}
\forall\,\delta_1,\ldots,\delta_{\ell}\in\N:\,\forall\,n\geq N: \abs{\E a_n(P_1)^{\delta_1}\cdots a_n(P_{\ell})^{\delta_{\ell}}}\ &\leq\ \frac{C(\delta_1 + \ldots + \delta_{\ell})}{n^{\frac{1}{2}\#\{i\in\oneto{\ell}|\delta_i=1\}}}\label{eq:AU1}\tag{AU1}\\
\forall\, n\geq N: \bigabs{\E a_n(P_1)^2\cdots a_n(P_{\ell})^2 - 1}\ &\leq\ C^{(\ell)}_n	\label{eq:AU2}\tag{AU2}
\end{align}
We will call a sequence $(a_n)_n$ with the properties above an \emph{approximately uncorrelated triangular scheme} (cf.\ \cite{Hochstattler:Kirsch:Warzel:2016}, \cite{Fleermann:Kirsch:Kriecherbauer:2021}). Note that \eqref{eq:AU1} entails that all random variables $a_n(p,q)$ have uniformly bounded absolute moments of all orders, for example, $\E\abs{a_n(p,q)}^k\leq (\E (a_n(p,q))^{2k})^{1/2}\leq C(2k)^{1/2}$. Further, by \eqref{eq:AU1} and \eqref{eq:AU2}, the entries $a_n(p,q)$ need not have zero expectation nor unit variance, but these properties do hold asymptotically.

For all $n\in\N$, each element $b_n\in\{b\in\N\,|\, 1\leq b <n,\, b\text{ odd}\}\cup\{n\}$ will be called \emph{$(n)$-bandwidth}. If $b_n$ is a bandwidth, we call an index pair $(p,q)\in\oneto{n}^2$ \emph{$b_n$-relevant}, if $b_n=n$ or $\abs{p-q}\leq(b_n-1)/2$ or $\abs{p-q}\geq n - (b_n-1)/2$. In particular, if $b_n=n$, then all pairs $(p,q)\in\oneto{n}^2$ are $b_n$-relevant, but if $b_n< n$, $(p,q)$ is $b_n$-relevant iff
\[
\abs{p-q} \leq \frac{b_n-1}{2} \qquad\text{or}\qquad \abs{p-q} \geq n - \frac{b_n-1}{2}.
\] 
If $b_n$ is a bandwidth, then an $n\times n$ matrix $X$ will be called periodic band matrix with bandwidth $b_n$ if $X(p,q)=0$ for all $(p,q)\in\oneto{n}^2$ which are not $b_n$-relevant. For example, an $8\times 8$ periodic band matrix $X$ with bandwidth $5$ has the structure
\begin{equation*}
X=
\begin{pmatrix}
x_{1,1} & x_{1,2} &	x_{1,3}	& 0	& 0	& 0	& x_{1,7} & x_{1,8} \\
x_{2,1} & x_{2,2} & x_{2,3}	& x_{2,4} &	0 & 0 & 0 &	x_{2,8} \\
x_{3,1}	& x_{3,2} & x_{3,3}	& x_{3,4} &	x_{3,5}	& 0 & 0 & 0	\\
0 & x_{4,2} & x_{4,3} & x_{4,4} & x_{4,5} & x_{4,6} & 0 & 0 \\
0 & 0 & x_{5,3} & x_{5,4} & x_{5,5} 	& x_{5,6} &	x_{5,7} & 0 \\
0 & 0 & 0 & x_{6,4} & x_{6,5} &	x_{6,6} 	& x_{6,7} &	x_{6,8} 	\\
x_{7,1} & 0 & 0 & 0 & x_{7,5} & x_{7,6} & x_{7,7} &	x_{7,8} 	\\
x_{8,1} & x_{8,2} & 0 & 0 & 0 & x_{8,6} & x_{8,7} &	x_{8,8} 	\\
\end{pmatrix}.
\end{equation*}
The bandwidth can be interpreted as the number of allowable non-trivial entries in the "middle row" of the matrix. Given an approximately uncorrelated triangular scheme $(a_n)_n$ as above and bandwidths $(b_n)_n$, we define the periodic band matrices $a_n^b$ by
\[
\forall\, n\in\N:\,\forall\, P\in\oneto{n}^2:\, a_n^b(P) \defeq \begin{cases}
 a_n(P) &\quad \text{if $P$ is $b_n$-relevant,}\\
 0	& \quad\text{otherwise.}
 \end{cases}
 \]
 Finally, if $(a_n)_n$ is an approximately uncorrelated triangular scheme, $(b_n)_n$ is a sequence of bandwidths, we say that a sequence of periodic random band matrices $(X_n)_n$ \emph{is based on the triangular scheme $(a_n)_n$ with bandwidths $(b_n)_n$}, if for all $n\in\N$, $X_n = \frac{1}{\sqrt{b_n}}a_n^b$.
 Note that by choosing the bandwidth $b_n=n$ we obtain $X_n=\frac{1}{\sqrt{n}}a_n$, so the case of full matrices is covered by our setup. The first contribution of this paper is the following theorem, which generalizes Theorem 2.3 in \cite{Fleermann:Kirsch:Kriecherbauer:2021}, since we obtain a stronger result with fewer assumptions.
 \begin{theorem}
 \label{thm:bandmatrix}
 Let $(X_n)_n$ be a sequence of periodic random band matrices which are based on an approximately uncorrelated triangular scheme $(a_n)_n$ with bandwidths $(b_n)_n$, where $b_n\to\infty$. Then the following statements hold:
 \begin{enumerate}[i)]
 \item The semicircle law holds for $(X_n)_n$ in probability.
 \item If there exists a $p\in \N$ such that $b_n^{-p}$ is summable, and if for all $\ell\in\N$, the sequences $C^{(\ell)}_n$ are summable, then the semicircle law holds for $(X_n)_n$ almost surely.
 \end{enumerate}
 \end{theorem}
 \begin{remark}
 In Theorem~\ref{thm:bandmatrix} $ii)$, if all entries in each $a_n$ are $\{\pm 1\}$-valued, the condition on the summability of the sequences $C^{(\ell)}_n$ can be dropped, since then in \eqref{eq:AU2}, $C^{(\ell)}_n\defeq 0$ is a valid choice for these constants. Further, we point out that in particular, all bandwidths $b_n$ with $b_n \sim n^{\delta}$ for some $\delta\in(0,1)$ are covered by Theorem~\ref{thm:bandmatrix} $ii)$.
 \end{remark}

\subsection{Examples} We first note that Wigner matrices fit our framework, since their non-standardized versions satisfy \eqref{eq:AU1} and \eqref{eq:AU2}. For the independent case, statements stronger than those in Theorem~\ref{thm:bandmatrix} are known, see \cite{Fleermann:Kirsch:Kriecherbauer:2021} and references therein. The main contribution of Theorem~\ref{thm:bandmatrix} is thus for the truly dependent case. Therefore, we will now introduce two prominent ensembles of random matrices with dependent (even correlated) entries that fit our framework. These models will also serve as examples for the second part of this paper.

\subsubsection{Curie-Weiss ensembles.} 
\label{sec:curieweiss}

\begin{definition}\label{def:curieweiss}
Let $n\in\N$ be arbitrary and $Y_1,\ldots,Y_n$ be random variables defined on some probability space $(\Omega,\Acal,\Prob)$. Let $\beta>0$, then we say that $Y_1,\ldots,Y_n$ are \emph{Curie-Weiss($\beta$,$n$)-distributed\label{sym:CurieWeissdist}}, if for all $y_1,\ldots,y_n\in\{-1,1\}$ we have that
\[
\Prob(Y_1=y_1,\ldots,Y_n=y_n) = \frac{1}{Z_{\beta,n}}\cdot e^{\frac{\beta}{2n}\left(\sum y_i\right)^2},
\]
where $Z_{\beta,n}$\label{sym:CWconstant} is a normalization constant. The parameter $\beta$ is called \emph{inverse temperature}.
\end{definition}

The Curie-Weiss($\beta,n$) distribution is used to model the behavior of $n$ ferromagnetic particles at inverse temperature $\beta$. At large values of $\beta$ (low temperatures), all spins are likely to have the same alignment, modeling strong magnetization. At high temperatures, however, the spins can act almost independently, resembling weak magnetization. We refer the reader to \cite{Kirsch:2015} for a self-contained treatment of the Curie-Weiss distribution.

\begin{definition}
\label{def:CWensemble}
Let $0<\beta\leq 1$ and let the random variables $(\tilde{a}_n(i,j))_{1\leq i,j \leq n}$ be Curie-Weiss($\beta,n^2$)-distributed for each $n\in\N$ . Define the triangular scheme $(a_n)_n$ by setting
\[
\forall\,n\in\N:\,\forall\, (i,j)\in\oneto{n}^2:~a_n(i,j) =
\begin{cases}
\tilde{a}_n(i,j) & \text{if $i\leq j$}\\
\tilde{a}_n(j,i) & \text{if $i> j$}.
\end{cases}
\]
Then $(a_n)_n$ will be called \emph{Curie-Weiss($\beta$) ensemble}.	
\end{definition}

\begin{corollary}\label{cor:CurieWeiss}
Let $(a_n)_n$ be a Curie-Weiss ensemble at inverse temperature $\beta\in(0,1]$. Let $(b_n)_n$ be a sequence of $n$-bandwidths and $(X_n)_n$ be the periodic random band matrices which are based on $(a_n)_n$ with bandwidths $(b_n)_n$. Then the following statements hold:
 \begin{enumerate}[i)]
 \item If $b_n\to\infty$, then the semicircle law holds for $(X_n)_n$ in probability.
 \item If for some $p\in\N$, $\frac{1}{b_n^p}$ is summable over $n$, then the semicircle law holds almost surely for $(X_n)_n$.
 \end{enumerate}

\end{corollary}
\begin{proof}
In \cite{Hochstattler:Kirsch:Warzel:2016} and \cite{Kirsch:2015}, it was shown that that $(a_n)_n$ is approximately uncorrelated with constants $C^{(\ell)}_n\defeq 0$ and $C(\ell)\defeq (\ell-1)!!$ for $\beta\in(0, 1/2]$, $C(\ell)\defeq (\ell-1)!!(\beta/(1-\beta))^{\ell/2}$ for $\beta\in(1/2,1)$ and  $C(\ell)\defeq 12^{\ell/4}\Gamma((\ell+1)/4)/\Gamma(1/4)$ for $\beta=1$, cf.\ Theorem 5.17 in \cite{Kirsch:2015}. Therefore, the statement follows with Theorem~\ref{thm:bandmatrix}.
\end{proof}

\subsubsection{Correlated Gaussian ensembles.}
\label{sec:corgauss}

\noindent We impose the following assumptions on the triangular scheme $(a_n)_n$, where due to symmetry, we need only specify the upper right triangle:
\begin{enumerate}
\item $\forall\, n\in\N: (a_n(p,q))_{1\leq p\leq q\leq n}\sim \Ncal(0,\Sigma_n)$, where $\Sigma_n$ is a positive definite $(n(n+1)/2)\times(n(n+1)/2)$ covariance matrix, indexed by pairs $(p,q),(r,s)$ with $1\leq p\leq q \leq n$ and $1\leq r\leq s \leq n$. In this case, we also set $\Sigma_n((q,p),(r,s)) \defeq \Sigma_n((p,q),(r,s))$ and similarly if the second pair (or both pairs) is flipped to account for the symmetry of $a_n$. 
\item We assume the sequence $(\Sigma_n)_n$ satisfies for all $n\in\N$ and all fundamentally different pairs $(p,q),(r,s)\in\oneto{n}^2$:
\begin{align}
	& \abs{\E a_n(p,q)a_n(r,s)} = \abs{\Sigma_n((p,q),(r,s))} \leq \frac{1}{n}.\label{eq:corrdecay}\\
	& \E (a_n(p,q))^2 = \abs{\Sigma_n((p,q),(p,q))} = 1. \label{eq:unitvar}
\end{align}

\end{enumerate}
Then we call $(a_n)_n$ an \emph{approximately uncorrelated Gaussian ensemble}. For clarity reasons, the setup we chose here differs from the expositions in \cite{Fleermann:2019} and \cite{Fleermann:Kirsch:Kriecherbauer:2021}. Note also that indexing the covariance matrices by pairs leads to no significant ambiguity, since positive definiteness and conditions \eqref{eq:corrdecay} and \eqref{eq:unitvar} are not affected by the index order. But with this new indexing, we are able to give a very brief proof of the following lemma (cf.\ Lemmas 3.9 and 3.10 in \cite{Fleermann:Kirsch:Kriecherbauer:2021}), which we include in Appendix~\ref{sec:ProofAUGaussian} for the convenience of the reader.
\begin{lemma}
\label{lem:AUGaussian}
	If $(a_n)_n$ is an approximately uncorrelated Gaussian ensemble, then $(a_n)_n$ is an approximately uncorrelated triangular scheme. More precisely it satisfies conditions \eqref{eq:AU1} with constants $C(\ell)\defeq\Pcal\Pcal(\ell)$ and \eqref{eq:AU2} with sequences $C^{(\ell)}_n\defeq \Pcal\Pcal(2\ell)/n^2$, which are summable over $n$.
\end{lemma}
\begin{corollary}\label{cor:Gaussian}
Let $(a_n)_n$ be an approximately uncorrelated Gaussian ensemble. Let $(b_n)_n$ be a sequence of $n$-bandwidths and $(X_n)_n$ be the periodic random band matrices which are based on $(a_n)_n$ with bandwidths $(b_n)_n$. Then the following statements hold:
 \begin{enumerate}[i)]
 \item If $b_n\to\infty$, then the semicircle law holds for $(X_n)_n$ in probability.
 \item If for some $p\in\N$, $\frac{1}{b_n^p}$ is summable over $n$, then the semicircle law holds almost surely for $(X_n)_n$. Observe that this statement applies in particular to full matrices ($b_n=n$), where $p=2$ may be chosen.
 \end{enumerate}
\end{corollary}
\begin{proof}
This is immediate with Lemma~\ref{lem:AUGaussian} and Theorem~\ref{thm:bandmatrix}	
\end{proof}

\subsection{Proof of Theorem~\ref{thm:bandmatrix}}
Denote by $(\sigma_n)_n$ the ESDs of $(X_n)_n$ and by $\sigma$ the semicircle distribution. In \cite{Fleermann:Kirsch:Kriecherbauer:2021} the strategy of proof was to show that
\begin{align}
\forall\, k\in\N:\ \E\integrala{\sigma_n}{x^k} & \ \xrightarrow[n\to\infty]{}\ \integrala{\sigma}{x^k},\label{eq:expconv}\\
\forall\, k\in\N: \V\integrala{\sigma_n}{x^k} & \ \xrightarrow[n\to\infty]{}\ 0\label{eq:varconv}.
\end{align}
Then \eqref{eq:expconv} entails weak convergence in expectation and together with \eqref{eq:varconv} this ensures weak convergence in probability. Further, if in \eqref{eq:varconv} the decay is summably fast, then this entails weak convergence almost surely (see e.g.\ \cite{Fleermann:2019} for details).

However, in \cite{Fleermann:2015} it was shown (and also presented in \cite{Fleermann:2019}) that $\sigma_n\to\sigma$ weakly in expectation resp.\ in probability resp.\ almost surely if for all $k\in\N:$ $\integrala{\sigma_n}{x^k}\to\integrala{\sigma}{x^k}$ in expectation resp.\ in probability resp.\ almost surely.
Using this observation, we can refine the above proof strategy as follows: For each $k\in\N$ we identify a number $\ell\in\N$ and a decomposition 
\begin{equation}
\label{eq:decomposition}
\integrala{\sigma_n}{x^k} \ = \ D_n^{(1)} + D_n^{(2)} + \ldots + D_n^{(\ell)}
\end{equation}
into finitely many summands (where $\ell$ is independent of $n$), so that these summands are amenable for analysis individually. For example, if we can prove that  
\begin{align}
\forall\, 1\leq i \leq \ell-1 &:\ \E D_n^{(i)} \xrightarrow[n\to\infty]{} 0  \quad \text{and}\quad \E D_n^{(\ell)} \xrightarrow[n\to\infty]{} \Cat_{\frac{k}{2}}\one_{2\N}(k) \label{eq:rexpconv}\\
\forall\, 1\leq i \leq \ell:\, \exists\, z\in\N &:\ \E\abs{D_n^{(i)}-\E D_n^{(i)}}^z \xrightarrow[n\to\infty]{} 0 \label{eq:rvarconv},
\end{align}
where $\Cat_{\ell}$ is the $\ell$-th Catalan number, then \eqref{eq:rexpconv} secures convergence in expectation of $\integrala{\sigma_n}{x^k}$ to $\integrala{\sigma}{x^k}$ and together with \eqref{eq:rvarconv} this ensures convergence in probability of $\integrala{\sigma_n}{x^k}$ to $\integrala{\sigma}{x^k}$. Further, if all decays in \eqref{eq:rvarconv} are summably fast, then we obtain convergence almost surely of $\integrala{\sigma_n}{x^k}$ to $\integrala{\sigma}{x^k}$.
This strategy -- although seemingly more complicated -- has crucial advantages over the strategy pertaining to \eqref{eq:expconv} and \eqref{eq:varconv}. First, note that a direct analysis of $\V\integrala{\sigma_n}{x^k}= \V(D_n^{(1)}+\ldots +D_n^{(\ell)})$ requires the analysis of covariances $\Cov(D_n^{(i)},D_n^{(j)})$ for $i\neq j$. These mixed terms were responsible for serious problems in the analysis of \cite{Fleermann:Kirsch:Kriecherbauer:2021}, see Step 1/Case 1/Subcase 2 in the proof of their Theorem 2.3, for example, which led to further required conditions (e.g. (AAU3) in \cite{Fleermann:Kirsch:Kriecherbauer:2021}). The summand-wise approach merely requires the analysis of the terms $\E D_n^{(i)}$, $\V D_n^{(i)}$, $i=1\ldots,\ell$. The second advantage is that the individual summands $D_n^{(i)}$ -- as we will see -- can be chosen so that arbitrarily high central moments of $D_n^{(i)}$ are amenable for analysis (this was unwieldy when starting from $\integrala{\sigma_n}{x^k}$ due to the multitude of mixed terms emerging for high central moments). This high-moment analysis in turn allows us to lessen our requirements on the decay rate of the bandwidth $(b_n)_n$. Note that using the former approach as in \cite{Fleermann:Kirsch:Kriecherbauer:2021}, a higher moment analysis was not just unwieldy, but also not promising, see Remark 4.12 in \cite{Fleermann:Kirsch:Kriecherbauer:2021}. But let us begin with the proof: We begin as usual by writing
\begin{equation}
\label{eq:sum}	
\integrala{\sigma_n}{x^k}=\frac{1}{n}\tr X_n^k = \frac{1}{nb_n^{\frac{k}{2}}}\sum_{\ubar{t}\in\oneto{n}^k} a_n^b(\ubar{t})
\end{equation}
where $\ubar{t}=(t_1,\ldots,t_k)$ and $a_n^b(\ubar{t}) = a_n^b(t_1,t_2)a_n^b(t_2,t_3)\cdots a_n^b(t_k,t_1)$.
Note that in \eqref{eq:sum}, whenever a pair $(t_{\ell},t_{\ell+1})$ is not \emph{$b_n$-relevant}, the summand vanishes. Thus, we call a tuple $\ubar{t}\in\oneto{n}^k$ $b_n$-relevant, if each pair $(t_{\ell},t_{\ell+1})$, $\ell=1,\ldots,k$, is $b_n$-relevant, and we set $\oneto{n}_b^k\defeq \{\ubar{t}\in\oneto{n}^k\, | \, \ubar{t} \text{ is } b_n\text{-relevant}\}$. We arrive at
\begin{equation}
\label{eq:sum2}	
\integrala{\sigma_n}{x^k}= \frac{1}{nb_n^{\frac{k}{2}}}\sum_{\ubar{t}\in\oneto{n}_b^k} a_n^b(\ubar{t}).
\end{equation}
Now we identify a tuple $\ubar{t}$ with its Eulerian graph $G_{\ubar{t}}=(V_{\ubar{t}},E_{\ubar{t}},\phi_{\ubar{t}})$ with vertices $V_{\ubar{t}} \defeq \{t_1,\ldots,t_k\}$, abstract edges $E_{\ubar{t}}=\{e_1,\ldots,e_k\}$ and incidence function $\phi_{\ubar{t}}: E_{\ubar{t}}\to\{U\subseteq V_{\ubar{t}}\, |\, \# U \in\{1,2\}\}$, where $\phi_{\ubar{t}}(e_{\ell})=\{t_{\ell},t_{\ell+1}\}$ for $\ell=1,\ldots,k$, where $k+1\equiv 1$. Denote $\pi_{\ell}(\ubar{t})\defeq \#\{\phi_{\ubar{t}}(e)\, | \, \text{$e$ is an $\ell$-fold edge in $\ubar{t}$}\}$, then we call the vector $\pi(\ubar{t})\defeq(\pi_1(\ubar{t}),\ldots,\pi_k(\ubar{t}))$ the \emph{profile} of $\ubar{t}$. Denote by $\Pi(k)\defeq\{\pi(\ubar{t})\,|\, \ubar{t}\in\oneto{n}^k,n\in\N\}$ the set of all possible profiles of $k$-tuples. Then clearly,
\begin{equation}
\label{eq:clearly}	
\#\Pi(k)\leq(k+1)^k \quad \text{and}\quad \forall\, \pi\in\Pi(k): k =\sum_{\ell=1}^k \ell\cdot\pi_{\ell}.
\end{equation}
For all $\pi\in\Pi(k)$, we define the following sets of tuples:
\begin{align*}
\Tcal_n(\pi)&\ \defeq \ \{\ubar{t}\in\oneto{n}_b^k\,|\, \pi(\ubar{t})=\pi\}\\
\Tcal_n^d(\pi)&\ \defeq \ \{(\ubar{s},\ubar{t})\in(\oneto{n}_b^k)^2\, | \, \pi(\ubar{t})=\pi=\pi(\ubar{s}),\ \phi_{\ubar{t}}(E_{\ubar{t}})\cap\phi_{\ubar{s}}(E_{\ubar{s}})=\emptyset\}\\
\Tcal_n^c(\pi)&\ \defeq \ \{(\ubar{s},\ubar{t})\in(\oneto{n}_b^k)^2\, | \, \pi(\ubar{t})=\pi=\pi(\ubar{s}),\ \phi_{\ubar{t}}(E_{\ubar{t}})\cap\phi_{\ubar{s}}(E_{\ubar{s}})\neq\emptyset\}\\
\Tcal_n^{c,\ell}(\pi)&\ \defeq \ \{(\ubar{s},\ubar{t})\in(\oneto{n}_b^k)^2\, | \, \pi(\ubar{t})=\pi=\pi(\ubar{s}),\ \#(\phi_{\ubar{t}}(E_{\ubar{t}})\cap\phi_{\ubar{s}}(E_{\ubar{s}}))=\ell\}
\end{align*}
where the last set is defined for all $\ell\in\oneto{k}$. To explain these sets, $\Tcal_n(\pi)$ contains all $\ubar{t}\in\oneto{n}^k_b$ with profile $\pi$, $\Tcal^d_n(\pi)$ contains all edge-disjoint pairs $(\ubar{s},\ubar{t})$, where $\ubar{s}$, $\ubar{t}\in\oneto{n}^k_b$ with profiles $\pi$, $\Tcal_n^c(\pi)$ contains all such tuple pairs which share at least one edge and $\Tcal_n^{c,\ell}(\pi)$ all those that share exactly $\ell$ edges.

 In the following, we will say that a $\pi\in\Pi(k)$ admits an odd edge resp.\ only even edges, if $\pi_{\ell}\geq 1$ for some $\ell\in\oneto{k}$ odd resp.\ if $\pi_{\ell}=0$ for all $\ell\in\oneto{k}$ odd. The following lemma holds (see \cite{Fleermann:2019} Lemmas 4.31, 4.33, 4.34 and 4.37.)

\begin{lemma}
\label{lem:tuplesetcount}
Let $n,k\in\N$ and let $\pi\in\Pi(k)$ be arbitrary.
\begin{enumerate}[A)]
\item Let $\ubar{t} \in \oneto{n}^k$ be arbitrary, then
	\begin{enumerate}[i)]
	\item $\# V_{\ubar{t}} \leq 1 + \pi_1(\ubar{t}) + \ldots \pi_k(\ubar{t})$.
	\item If $\ubar{t}$ contains at least one odd edge, then $\# V_{\ubar{t}} \leq \pi_1(\ubar{t}) + \ldots \pi_k(\ubar{t})$.
	\end{enumerate}
\item Let $b_n$ be a bandwidth and $\ell\in\oneto{k}$, then $\#\{\ubar{t}\in\oneto{n}_b^k\,|\, \#V_{\ubar{t}}\leq\ell\} \leq k^knb_n^{\ell-1}$.

\item Let $\pi\in\Pi(k)$ be arbitrary, then
	\begin{enumerate}[i)]
	\item $\#\Tcal_n(\pi)\ \leq\ k^k n b_n^{\pi_1+\ldots + \pi_k}$.
	\item If $\pi$ admits an odd edge, then $\#\Tcal_n(\pi)\ \leq\ k^k n b_n^{\pi_1+\ldots + \pi_k -1}$.
	\item  $\#\Tcal_n^d(\pi)\ \leq\ (\#\Tcal_n(\pi))^2$.
	\item If $\pi$ admits only even edges, then  $\#\Tcal_n^c(\pi)\leq k^2 (2k)^{2k} n b_n^{k-1}$.
	\item If $\pi$ admits at least one odd edge, then $ \#\Tcal_n^c(\pi)\ \leq\ k^2 (2k)^{2k} n b_n^{2(\pi_1+\ldots + \pi_k) -2}$.
	\item If $\pi$ admits at least one odd edge, then we have for all $\ell=1,\ldots,k$:
	\[
	\#\Tcal_n^{c,\ell}(\pi)\ \leq\ k^2 (2k)^{2k} n b_n^{2(\pi_1+\ldots + \pi_k) -\ell-1}
	\]
\end{enumerate}
\end{enumerate}	
\end{lemma}
We now sort the sum in \eqref{eq:sum2} according to profiles $\pi\in\Pi(k)$ and get
\begin{equation}
\label{eq:sum3}
\integrala{\sigma_n}{x^k}= \sum_{\pi\in\Pi(k)}\frac{1}{nb_n^{\frac{k}{2}}}\sum_{\ubar{t}\in\Tcal_n(\pi)} a_n^b(\ubar{t}).	
\end{equation}
We have now achieved a finite decomposition (since $\Pi(k)$ is a finite set) as in \eqref{eq:decomposition} and will proceed as outlined in \eqref{eq:rexpconv} and \eqref{eq:rvarconv}, that is, for each $\pi\in\Pi(k)$, we will analyze the expectation and variance of
\begin{equation}
\label{eq:toanalyze}
\frac{1}{nb_n^{\frac{k}{2}}}\sum_{\ubar{t}\in\Tcal_n(\pi)} a_n^b(\ubar{t}).
\end{equation}
\noindent\underline{Case 1: $\pi$ admits only even edges.}\newline
\noindent\underline{Subcase 1: $\pi = (0,k/2,0,\ldots,0)$.}\newline
In this subcase, each $\ubar{t}\in\Tcal_n(\pi)$ consists of $k/2$ double edges. By Lemma~\ref{lem:tuplesetcount}, a $\ubar{t}$ in this set has at most $k/2+1$ vertices. We construct the more detailed subsets
\[
\Tcal_n^{\leq k/2}(\pi) \defeq \{\ubar{t}\in \Tcal_n(\pi)\, | \, \#V_{\ubar{t}}\leq k/2 \} \quad\text{and}\quad \Tcal_n^{k/2+1}(\pi) \defeq \{\ubar{t}\in \Tcal_n(\pi)\, | \, \#V_{\ubar{t}}= k/2+1 \}.
\]
Then $\Tcal_n^{\leq k/2}(\pi)\leq k^k nb_n^{\frac{k}{2}-1}$ by Lemma~\ref{lem:tuplesetcount}. Thus,
\[
\frac{1}{n b_n^{\frac{k}{2}}} \sum_{\ubar{t}\in\Tcal_n^{\leq k/2}(\pi)} a_n^b(\ubar{t}) \xrightarrow[n\to\infty]{} 0
\]
in expectation and probability, and also almost surely if $b_n^{-p}$ is summable for some $p\in\N$ by Lemma~\ref{lem:breakaway}. Further, by \cite{Fleermann:2019} p.81 f.,
\[
\frac{1}{n b_n^{\frac{k}{2}}} \sum_{\ubar{t}\in\Tcal_n^{k/2+1}(\pi)}a_n^b(\ubar{t}) \xrightarrow[n\to\infty]{} \Cat_{\frac{k}{2}}.
\]
in expectation, and the variance of each of this sum is clearly upper bounded by
\begin{align}
&\frac{1}{n^2 b_n^k} \sum_{(\ubar{s},\ubar{t})\in\Tcal_n^d(\pi)}\abs{\E a_n^b(\ubar{s})a^b_n(\ubar{t})-\E a_n^b(\ubar{s})\E a^b_n(\ubar{t})}\label{eq:catvar1}\\
&+\ \frac{1}{n^2 b_n^k} \sum_{(\ubar{s},\ubar{t})\in\Tcal_n^c(\pi)}\abs{\E a_n^b(\ubar{s})a_n^b(\ubar{t})-\E a_n^b(\ubar{s})\E a^b_n(\ubar{t})}\label{eq:catvar2}.
\end{align}
Considering $\# \Tcal_n^d(\pi)\leq k^{2k}n^2b_n^k$ by Lemma~\ref{lem:tuplesetcount}, the sum in \eqref{eq:catvar1} converges to zero by \eqref{eq:AU1} and \eqref{eq:AU2}, and the decay is summably fast if $C^{(k/2)}_n$ decays to zero summably fast.
Next, considering $\# \Tcal_n^c(\pi)\leq k^2(2k)^{2k}n b_n^{k-1}$ by Lemma~\ref{lem:tuplesetcount}, the sum \eqref{eq:catvar2} converges to zero, and this convergence is summably fast if $(n b_n)^{-1}$ is summable, which is the case if $b_n^{-p}$ is summable for some $p\in\N$ (Young's inequality).

\noindent\underline{Subcase 2: $\pi_{\ell}\geq 1$ for some $\ell\geq 4$.}\newline
 Then since $\pi_1+\ldots+\pi_k \leq 1 + (k-4)/2 = k/2-1$, we obtain the bound $\#\Tcal_n(\pi)\leq k^k n b_n^{k/2-1}$ by Lemma~\ref{lem:tuplesetcount}. Thus,
\[
\frac{1}{nb_n^{\frac{k}{2}}}\sum_{\ubar{t}\in \Tcal_n(\pi)} a_n^b(\ubar{t}) \xrightarrow[n\to\infty]{} 0 
\]
in expectation and in probability, and also almost surely if $b_n^{-p}$ is summable for some $p\in\N$ (Lemma~\ref{lem:breakaway}). 

\noindent\underline{Case 2: $\pi_{\ell}\geq 1$ for some $\ell\in\N$ odd.}\newline
Then by Lemma~\ref{lem:tuplesetcount}, $\#\Tcal_n(\pi)\leq k^k n b_n^{\pi_1+\ldots +\pi_k -1}$. Further, by condition \eqref{eq:AU1},
\[
\forall\,\ubar{t}\in\Tcal_n(\pi):\,\abs{\E a_n^b(\ubar{t})}\leq\frac{C(k)}{n^{\frac{1}{2}\pi_1}},
\]
so
\begin{equation}
\label{eq:standard}	
\frac{1}{nb_n^{\frac{k}{2}}}\sum_{\ubar{t}\in\Tcal_n(\pi)}\abs{\E a_n^b(\ubar{t})} \leq \frac{k^kC(k)}{nb_n^{\frac{k}{2}}} \cdot \frac{nb_n^{\pi_1 + \ldots + \pi_k -1}}{b_n^{\frac{1}{2}\pi_1}} \leq \frac{k^kC(k)}{b_n},
\end{equation}
using $\pi_1/2 + \pi_2 + \ldots + \pi_k\leq k/2$ by \eqref{eq:clearly}. Since $b_n\to\infty$, convergence in expectation to zero follows.
Next, instead of analyzing the variance of the sum in question, we analyze an arbitrary high even central moment. To this end, let $z\in 2\N$ be arbitrary, then
\begin{align}
&\E\left(\frac{1}{nb_n^{\frac{k}{2}}}\sum_{\ubar{t}\in\Tcal_n(\pi)}(a_n^b(\ubar{t}) - \E a_n^b(\ubar{t}))  \right)^z \label{eq:zthmoment}\\
&=\ \bigabs{\frac{1}{n^zb_n^{\frac{zk}{2}}} \sum_{\ubar{t}^{(1)},\ldots, \ubar{t}^{(z)}\in\Tcal_n(\pi)} \E \prod_{s=1}^z \left[a_n^b(\ubar{t}^{(s)})- \E a_n^b(\ubar{t}^{(s)}) \right] }\notag\\
&\leq \ \sum_{S\subseteq \oneto{z}} \frac{1}{n^zb_n^{\frac{zk}{2}}} \sum_{\ubar{t}^{(1)},\ldots, \ubar{t}^{(z)}\in\Tcal_n(\pi)} \bigabs{\E \prod_{s\in S} a_n^b(\ubar{t}^{(s)})}\cdot \prod_{s\notin S} \abs{\E a_n^b(\ubar{t}^{(j)})}\notag\\
&= \ \sum_{S\subseteq \oneto{z}}\underbrace{\frac{1}{n^{\# S}b_n^{\frac{k\# S}{2}}} \sum_{\substack{\ubar{t}^{(i)}\in\Tcal_n(\pi)\\ i\in S}} \bigabs{\E \prod_{i\in S} a_n^b(\ubar{t}^{(i)})}}_{=:A_n(S)} \cdot \underbrace{\frac{1}{n^{\# S^c}b_n^{\frac{k\# S^c}{2}}}\sum_{\substack{\ubar{t}^{(j)}\in\Tcal_n(\pi)\\ j\in S^c}} \prod_{j\in S^c} \abs{\E a_n^b(\ubar{t}^{(j)})}}_{=:B_n(S)},\notag
\end{align}
where for any $S\subseteq\oneto{z}$, $S^c\defeq\oneto{z}\backslash S$.
We will analyze the terms $A_n(S)$ and $B_n(S)$ separately. Notationally, we write $s\defeq\# S$ and $s^c\defeq \# S^c$ and note that $s+s^c=z$. For $B_n(S)$ we find an upper bound using \eqref{eq:standard}:
\[
B_n(S)\ \leq \  \frac{(k^kC(k))^{s^c}}{b_n^{s^c}}.
\]
For $A_n(S)$ we need to account for common edges among the $\ubar{t}^{(i)}$, $i\in S$. These have the effect that on the one hand, overlaps lead to fewer possible tuples, but on the other hand, single edges might overlap, negating a possible decay which existed due to \eqref{eq:AU1}. If $S=\emptyset$, then $A_n(S)=1$ as an empty product. If $S=\{i\}$ for some $i\in\oneto{z}$, then by \eqref{eq:standard},
\begin{equation}
\label{eq:Ssmall}
A_n(S) \leq \frac{k^k C(k)}{b_n}, \quad\text{so}\quad A_n(S)B_n(S)\leq \frac{(k^kC(k))^z}{b_n^z}
\end{equation}
in both cases $\# S = 0$ and $\# S =1$. We now assume that $s=\# S\geq 2$. Write $S=\{i_1,\ldots, i_s\}$, where $i_j\in\oneto{z}$ and $i_1<\ldots<i_s$. Then for all $\ell_2,\ldots,\ell_s\in\{0,1,\ldots,\sum_i\pi_i\}$ we denote by $\Tcal^{(\ubar{\ell})}_n(\pi)$ the set of tuples $(\ubar{t}^{(i_1)},\ldots,\ubar{t}^{(i_s)})$, where $\ubar{t}^{(i_j)}\in\Tcal_n(\pi)$ and 
\[
\forall\, j\in\{2,\ldots,s\}:\ \#\phi_{\ubar{t}^{(i_j)}}(E_{\ubar{t}^{(i_j)}}) \cap \bigcup_{1\leq r < j}\phi_{\ubar{t}^{(i_r)}}(E_{\ubar{t}^{(i_r)}}) = \ell_j,
\]
in words, each $\ubar{t}^{(i_j)}$ has exactly $\ell_j$ of its $\sum_i\pi_i$ different edges in common with previous tuples $\ubar{t}^{(i_r)}$, $r<j$. Then
\[
A_n(S)= \sum_{\ubar{\ell}\in \{0,\ldots,\sum_i\pi_i\}^{s-1}}\underbrace{\frac{1}{n^sb_n^{\frac{ks}{2}}} \sum_{(\ubar{t}^{(i_j)})_{j\in\oneto{s}}\in\Tcal_n^{(\ubar{\ell})}(\pi)} \bigabs{\E \prod_{j=1}^s a_n^b(\ubar{t}^{(i_j)})}}_{=:A_n^{(\ubar{\ell})}(S)}.
\]
If $\ell=(0,\ldots,0)$, this entails that all tuples are edge-disjoint, so we use the trivial upper bound
\[
\# \Tcal_n^{(\ubar{\ell})}(\pi)\ \leq\ (\#\Tcal_n(\pi))^s\ \leq\ (k^k n b_n^{\pi_1+\ldots + \pi_k-1})^s,
\]
and since no single edges may vanish due to overlaps, we obtain by \eqref{eq:AU1}
\begin{equation}
\label{eq:expprod}	
\bigabs{\E \prod_{j=1}^s a_n^b(\ubar{t}^{(i_j)})} \leq \frac{C(k s)}{n^{\frac{1}{2}\pi_1 s}},
\end{equation}
so
\[
A_n^{(\ubar{\ell})}(S) \leq  \frac{1}{n^s b_n^{\frac{k}{2}s}} (k^k n b_n^{\pi_1+\ldots + \pi_k-1})^s \cdot \frac{C(ks)}{n^{\frac{1}{2}\pi_1s}}
 \leq \frac{k^{ks}C(ks)}{b_n^s},
\]
and subsequently
\[
A_n^{(\ubar{\ell})}(S)\cdot B_n(S) \leq \frac{k^{ks}C(ks)}{b_n^s} \cdot \frac{(k^kC(k))^{s^c}}{b_n^{s^c}} = \frac{k^{zk}C(k)^{s^c}C(ks)}{b_n^z}.
\]
If $\ubar{\ell}\neq (0,\ldots,0)$, then some of the tuples in $(\ubar{t}^{(i_j)})_{j\in\oneto{s}}$ have common edges. As mentioned, this has two effects: On the one hand, the corresponding product on the l.h.s.\ of \eqref{eq:expprod} cannot be guaranteed to decay at the speed given on the r.h.s.\ of \eqref{eq:expprod}, since single edges in different tuples might overlap, negating the decay effect guaranteed by (AU1). To be more precise, for each overlap at most two single edges may be eradicated, leading to at least $s\cdot\pi_1 - 2\ell_2-2\ell_3\ldots-2\ell_s$ remaining single random variables in the product on the l.h.s.\ of \eqref{eq:expprod}. So if $\ubar{\ell}\neq (0,\ldots,0)$, then \eqref{eq:expprod} becomes
\begin{equation}
\label{eq:expprodoverlap}	
\bigabs{\E \prod_{j=1}^s a_n^b(\ubar{t}^{(i_j)})} \leq \frac{C(ks)}{n^{\frac{1}{2}\max(s\pi_1-2\ell_2-\ldots-2\ell_s,0)}}
\end{equation}

The second effect is that overlaps of edges entail fewer possible vertices, which decreases the count $\#\Tcal_n^{(\ubar{\ell})}(\pi)$:

\begin{lemma}
\label{lem:overlapsetcount}
Let $k,n\in\N$ and $\pi\in\Pi(k)$ be arbitrary so that $\pi_i\neq 0$ for some $i\in\oneto{k}$ odd. Further, let $\ubar{\ell}\in\{0,\ldots,\sum_i\pi_i\}^{s-1}$ be arbitrary for some $s\geq 2$, where $\ubar{\ell}\neq(0,\ldots,0)$. Then
\[
\#\Tcal_n^{(\ell)}(\pi) \leq k^k n b_n^{\sum_i\pi_i -1} \cdot \prod_{\substack{i\in\{2,\ldots,s\}\\ \ell_i=0}} k^k n b_n^{\sum_i\pi_i -1} \cdot \prod_{\substack{i\in\{2,\ldots,s\}\\ \ell_i\geq 1}} (2k)^k k(ks)^{\ell_i}b_n^{\pi_1+\ldots+\pi_k-\ell_i}.
\]
\end{lemma}
\begin{proof}
The strategy of the proof is to derive an upper bound on the number of possibilities to construct an element $(\ubar{t}^{(1)},\ldots,\ubar{t}^{(s)})\in\Tcal^{(\ubar{\ell})}_n(\pi)$.

We proceed step by step: For $\ubar{t}^{(1)}$ we have at most $k^knb_n^{\pi_1+\ldots+\pi_k-1}$ possibilities by Lemma~\ref{lem:tuplesetcount}. Of course, the bound also applies for all $\ubar{t}^{(i)}$ with $i\geq 2$, but we will only use it again for those $\ubar{t}^{(i)}$, $i\geq 2$, for which $\ell_i=0$. 

 Now assume that $i\geq 2$ with $\ell_i\geq 1$. 
 Note that so far, at most $(i-1)\cdot k\leq sk$ vertices have been picked for previous tuples. Each tuple has at most $\pi_1+\ldots+\pi_k$ vertices by Lemma~\ref{lem:tuplesetcount}. We now obtain an upper bound on the \emph{new} vertices that $\ubar{t}^{(i)}$ may contain. To this end, 
 we start a cyclic tour along the tuple $\ubar{t}^{(i)}$, starting at an end point $t^{(i)}_{j+1}$ of an edge $e_j$ which is a common edge with a previously determined tuple. Then $t^{(i)}_{j+1}$ is not a newly observed vertex, and as we proceed cycicly along the tuple, we can observe at most $\pi_1+\ldots+\pi_k-\ell_i$ new vertices. 
 
 We now bound the possibilities to construct a $\ubar{t}\in\oneto{n}^k_b$ with at most $\pi_1+\ldots+\pi_k$ vertices, from which at most $\pi_1+\ldots+\pi_k-\ell_i$ vertices are new and all other vertices are old (i.e.\ appeared in some previous tuple), but at least one vertex must be old (since we have at least one overlap). Since there is at least one old vertex, we start with such a vertex for $t_1$ and in the end allow a cyclic permutation (e.g. $(1,2,3,4)\to (3,4,1,2)$) of the tuple to count all possibilities. For the construction of $\ubar{t}$ we first fix a map $f:\{1,\ldots,k\}\to\{1,\ldots,\pi_1+\ldots+\pi_k\}$ indicating which places in $\ubar{t}^{(i)}$ should receive equal or different vertices, i.e.\ $t_i=t_j:\Leftrightarrow f(i)=f(j)$. We assume that the coloring $f$ has standard form, that is, $f(1)=1$ and if $f(\ell)\notin\{f(1),\ldots,f(\ell-1)\}$, then $f(\ell)=\max\{f(1),\ldots,f(\ell-1)\} +1$. This choice of $f$ admits at most $k^k$ possibilities. Note that $\max f(\oneto{k})$ is the number of different vertices in $\ubar{t}$. To determine which of these should be new and which should be old, we fix another map $g:\{1,\ldots,\max f(\oneto{k})\}\to\{0,1\}$ with $g(1)=0$ and $\#\{i\in\{1,\ldots,\max f(\oneto{k})\}\,|\, g(1)=1\}\leq \pi_1+\ldots+\pi_k -\ell_i$, which admits at most $2^k$ possibilities. We then proceed as follows: For $t_1$ we choose an old vertex arbitrarily, yielding at most $(i-1)k\leq sk$ choices. Then if $t_1,\ldots,t_{\ell}$ have already been constructed, where $\ell\in\{1,\ldots,k-1\}$, we construct $t_{\ell+1}$ as follows: If $f(\ell+1)=f(\ell')$ for some $\ell'\in\{1,\ldots,\ell\}$, then $t_{\ell+1}$ must equal $t_{\ell'}$, leaving no choice for $t_{\ell+1}$. If $f(\ell+1)\notin\{f(1),\ldots,f(\ell)\}$, choose $t_{\ell+1}$ different from all previous vertices in $\ubar{t}$. If $g(f(\ell+1))=0$, pick an old vertex from some previous tuple, yielding at most $sk$ possibilities. If $g(f(\ell+1))=1$, pick a new vertex, yielding at most $b_n$ possibilities. This procedure to construct $\ubar{t}$ with help of $f$ and $g$ admits at most $(ks)^{\ell_i}
b_n^{\pi_1+\ldots +\pi_k - \ell_i}$
 possibilities.
 Cyclic permutation of this tuple admits at most $k$ choices, picking $f$ and $g$ at most $(2k)^k$ choices, so all together we had at most 
 $(2k)^k k(ks)^{\ell_i}b_n^{\pi_1+\ldots+\pi_k-\ell_i}$ choices for $\ubar{t}$. Since $\ubar{t}^{(i)}$ is of this form, this concludes the proof.
\end{proof}
Combining Lemma~\ref{lem:overlapsetcount} with \eqref{eq:expprodoverlap} and the prefactor in the definition of $A_n^{(\ubar{\ell})}(S)$ yields, setting $\ell(0)\defeq \#\{i\in\{2,\ldots,s\}\,|\,\ell_i=0\}$ and $\ell(\geq 1)\defeq \#\{i\in\{2,\ldots,s\}\,|\,\ell_i\geq 1\}$,
\begin{align*}
A_n^{(\ubar{\ell})}(S)&\leq \frac{1}{n^sb_n^{\frac{k}{2}s}} k^k n b_n^{\sum_i\pi_i -1} \cdot \prod_{\substack{i\in\{2,\ldots,s\}\\ \ell_i=0}} k^k n b_n^{\sum_i\pi_i -1} \\
&\quad\cdot \prod_{\substack{i\in\{2,\ldots,s\}\\ \ell_i\geq 1}} (2k)^{k}k(ks)^{\ell_i} b_n^{\sum_i\pi_i -\ell_i} \cdot \frac{C(ks)}{n^{\frac{1}{2}\max(s\pi_1-2\ell_2-\ldots-2\ell_s,0)}}\\
&\leq\frac{K(k,s)}{n^{s-\ell(0)-1}b_n^{\frac{k}{2}s}}b_n^{(\sum_i\pi_i-1)(\ell(0)+1)}\cdot b_n^{(\sum_i\pi_i)(\ell(\geq 1))}\frac{1}{b_n^{\max(s\pi_1/2,\, \ell_2+\ldots+\ell_s)}},
\end{align*}
where 
\[
K(k,s)\defeq k^{sk+1}(2k)^{ks}s^{ks}C(ks) \geq k^{k+k\ell(0)+1+k\ell(\geq 1)}(2k)^{k\ell(\geq 1)}s^{k\ell(\geq 1)}C(ks),
\]
for which we used $\ell(0)+\ell(\geq 1)+1 = s$. This inequality together with $\pi_1+\ldots+\pi_k\leq \pi_1 + (k-\pi_1)/2 = \pi_1/2 + k/2$ yields
\[
A_n^{(\ubar{\ell})}(S)\ \leq\ \frac{K(k,s)}{n^{s-\ell(0)-1}} \cdot\frac{b_n^{\frac{s\pi_1}{2}-\ell(0)-1}}{b_n^{\max(s\pi_1/2,\, \ell_2+\ldots+\ell_s)}}\leq \frac{K(k,s)}{b_n^s},
\]
and so
\begin{equation}
\label{eq:Slarge}
A_n^{(\ubar{\ell})}(S)B_n(S)\ \leq\ \frac{K(k,s)}{b_n^s}\cdot\frac{(k^kC(k))^{s^c}}{b_n^{s^c}} = \frac{K(k,s)(k^kC(k))^{s^c}}{b_n^z}.
\end{equation}
We have recognized the $z$-th central moment in \eqref{eq:zthmoment} as a finite sum of terms $A_n(S)B_n(s)$ and $A_n^{(\ubar{\ell})}(S)B_n(S)$, where the number of these terms is independent of $n$, and so that each summand decays at a speed of $b_n^{-z}$. This completes the proof of parts $i)$ and $ii)$ of Theorem~\ref{thm:bandmatrix} when choosing $z\geq p$ with $z\in 2\N$.

\subsection{Non-periodic band matrices}

In this section, we will see that Theorem~\ref{thm:bandmatrix} remains true for \emph{non-periodic} random matrices with approximately uncorrelated entries. To start, the concept of a bandwidth should be replaced by the concept called \emph{halfwidth}, which we adopted from \cite{Kirsch:Kriecherbauer:2018}. 
Roughly, the halfwidth $h=(h_n)_n$ is half of the bandwidth $b=(b_n)_n$, hence the name. 
Two $6\times 6$ non-periodic band matrices with halfwidths $2$ resp.\ $4$ have the structure
\[
\begin{pmatrix}
x_{1,1}		&		x_{1,2}		&			0				& 		0 			&			0				&	  0  \\
x_{2,1}		&		x_{2,2}		&		x_{2,3}		& 		0				&			0				&	 		0			 \\
	0				&		x_{3,2}		&		x_{3,3}		& 	x_{3,4} 	&			0				&	 		0		   \\
	0				&				0			&		x_{4,3}		& 	x_{4,4} 	&		x_{4,5}		&	 0  \\
	0				&				0			&			0				& 	x_{5,4} 	&		x_{5,5}		&	  x_{5,6}  \\
0	&				0			&			0				& 		0			 	&		x_{6,5}		&	  x_{6,6} 
\end{pmatrix},
\quad\text{resp.}\quad
\begin{pmatrix}
x_{1,1}		&		x_{1,2}		&		x_{1,3}		& 	x_{1,4} 	&		0   		&		0 		\\
x_{2,1}		&		x_{2,2}		&		x_{2,3}		& 	x_{2,4} 	&		x_{2,5}		&	 	0    	\\
x_{3,1}		&		x_{3,2}		&		x_{3,3}		& 	x_{3,4} 	&		x_{3,5}		&	  x_{3,6}  	\\
x_{4,1}		&		x_{4,2}		&		x_{4,3}		& 	x_{4,4} 	&		x_{4,5}		&	  x_{4,6}  	\\
	0		&		x_{5,2}		&		x_{5,3}		& 	x_{5,4} 	&		x_{5,5}		&	  x_{5,6}  	\\
	0		&		0			&		x_{6,3}		& 	x_{6,4} 	&		x_{6,5}		&	  x_{6,6} 	\\
\end{pmatrix}.
\]
The halfwidth should be interpreted as the number of allowable non-trivial entries in the first row of the matrix. As before, the bandwidth describes the number of such entries in the "middle row".

\begin{definition}
Let $n\in\N$ be arbitrary, then an $h_n\in\N$ is called \emph{($n$-)halfwidth}\label{sym:halfwidth}, if $h_n\in\{1\ldots n\}$. Given a sequence of halfwidths $h=(h_n)_n$, we set
\[
\forall\, n\in\N:~b_n\defeq \min(2 h_n-1,n)
\]
and call $b_n$ the \emph{bandwidth associated with the halfwidth} $h_n$.
\end{definition}

It is clear that for any $n\in\N$, the bandwidth $b_n$ that is associated with a halfwidth $h_n$ is either $n$ itself or an odd number in the set $\{1,\ldots,n\}$, thus coincides with the concept of a bandwidth in previous sections.

The difference between periodic and non-periodic matrices is that in the latter case, the triangular areas in the upper right and lower left corner of the matrices are missing, leading to the possibility that the \emph{inner band} is so wide that it reaches the top right and lower left corners of the matrix.

\begin{definition}
\label{def:bandmatrices}
Let $(\Omega,\Acal,\Prob)$ be a probability space, $(a_n)_{n\in\N}$ a triangular scheme, $(h_n)_n$ be a sequence of $n$-halfwidths with associated bandwidths $(b_n)_n$.
\begin{enumerate}
\item We define the non-periodic random matrices which are based on the triangular scheme $(a_n)_{n\in\N}$ with halfwidths $(h_n)_n$ as
\[
\forall\, n\in\N:\,\forall\,(i,j)\in\oneto{n}^2:~ X_n^{NP}(i,j)\defeq\label{sym:nonperrm}
\begin{cases}
\frac{1}{\sqrt{b_n}}a_n(i,j) & \mbox{if } \abs{i-j}\leq h_n-1,\\
0 & \mbox{otherwise.} 
\end{cases}
\]
\item We define the periodic random matrices which are based on the triangular scheme $(a_n)_{n\in\N}$ with associated bandwidth $b$ as
\[
\forall\, n\in\N:\,\forall\,(i,j)\in\oneto{n}^2:~ X_n^{P}(i,j)\defeq\label{sym:perrm}
\begin{cases}
\frac{1}{\sqrt{b_n}}a_n(i,j) & \mbox{if } \abs{i-j}\leq h_n-1, \\
\frac{1}{\sqrt{b_n}}a_n(i,j) & \mbox{if } \abs{i-j} \geq \max(n-h_n+1,h_n),\\
0 & \mbox{otherwise.} 
\end{cases}
\]
\end{enumerate}
\end{definition}

Note that the definition of periodic random matrices has not changed in comparison to previous sections as it is not hard to check that if $h_n$ is an $n$-halfwidth with associated bandwidth $b_n$, then an index pair $(p,q)\in\oneto{n}^2$ is $b_n$-relevant iff $\abs{p-q}\leq h_n-1$ or $\abs{p-q} \geq \max(n-h_n+1,h_n)$.

In \cite{Bogachev:2006} it was shown that for the i.i.d.\ case, the semicircle law holds in probability for $(X^{NP}_n)_n$ if
\begin{equation}
\label{eq:nonpercondition}
\lim_{n\to\infty} h_n=\infty\quad\text{and}\quad	
\lim_{n\to\infty}\frac{h_n}{n}\in\{0,1\},
\end{equation}
whereas the semicircle law does not hold if $\lim_n h_n/n=p$ for some $p\in(0,1)$. The analysis of this subsection derives the case \eqref{eq:nonpercondition} for the approximately uncorrelated setup. The case that $\lim_n h_n/n=p$ for some $p\in(0,1)$ is a corollary of our treatment in the second part of this paper.

Comparing non-periodic and periodic band matrices given some halfwidth $h_n$ and associated bandwidth $b_n$, we realize that both matrices contain a non-trivial area with indices $\abs{i-j}\leq h_n-1$, which is the band in the middle of the matrix, and additionally, periodic matrices contain non-trivial triangular areas with indices $\abs{i-j} \geq \max(n-h_n+1,h_n)$. Therefore,
the matrix $X_n^{P}-X_n^{NP}$ has rank $2\cdot\min(h_n,\,n-h_n +1)$.
We now use a well-known rank inequality (e.g. \cite{Fleermann:Kirsch:Kriecherbauer:2021}):
\begin{lemma}
\label{lem:rankinequality}
Let $Y$ and $\Ecal$ be real symmetric $n\times n$ matrices, where $\Ecal$ has rank $k$. Then it holds for all $z\in\C_+\defeq\{z\in\C\,|\,\Im(z)>0\}$:
\[
\bigabs{\frac{1}{n}\tr\left[(Y-z)^{-1}\right]-\frac{1}{n}\tr\left[(Y+\Ecal-z)^{-1}\right]} \leq \frac{2k}{n\Im(z)}.
\]
\end{lemma}
With this lemma, the following theorem follows directly from Theorem~\ref{thm:bandmatrix}:

\begin{theorem}\label{thm:nonperiodic}
Let $(a_n)_n$ be an approximately uncorrelated triangular scheme, $(h_n)_n$ a sequence of $n$-halfwidths and $(X^{NP}_n)_n$ the non-periodic random matrices which are based on $(a_n)_n$ with halfwidth $h$. We assume that
\[
h_n\to\infty \qquad \text{but} \qquad \lim_{n\to\infty}\frac{h_n}{n}\in \{0,1\}.
\]
Then we obtain the following results:
\begin{enumerate}[i)]
\item The semicircle law holds for $(X^{NP}_n)_n$ in probability. 
\item If $(1/h_n)^p_n$ is summable for some $p\in\N$, and the sequences $(C^{(\ell)}_n)_n$ from condition \eqref{eq:AU2} are summable for all $\ell\in\N$, then the semicircle law holds for $X_n^{NP}$ almost surely.
\end{enumerate}
\end{theorem}
\begin{proof}
Note that $h_n\to\infty \Leftrightarrow b_n\to\infty$ and for all $p\in\N$, $(h_n^{-p})_n$ is summable iff $(b_n^{-p})_n$ is summable. Therefore, under the conditions of Theorem~\ref{thm:nonperiodic} $i)$ resp.\ $ii)$, the conclusions of Theorem~\ref{thm:bandmatrix} $i)$ resp.\ $ii)$ hold.
It follows with the discussion before Lemma~\ref{lem:rankinequality} that $X_n^{P}-X_n^{NP}$ has rank $2\cdot\min(h_n,\,n-h_n +1)$. Therefore, if $s_n^{P}$ resp.\ $s_n^{NP}$ denote the Stieltjes transforms of the ESDs of $X_n^{P}$ resp.\ $X_n^{NP}$, then we find with Lemma~\ref{lem:rankinequality} that for $z\in\C_+$ arbitrary,
\[
\abs{s_n^P(z) - s_n^{NP}(z)} \leq \frac{2}{\Im(z)}\min\left(\frac{h_n}{n},\, 1 - \frac{h_n}{n} + \frac{1}{n}\right) \xrightarrow[n\to\infty]{} 0 \quad\text{surely.}
\] 
This concludes the proof, since under the conditions of Theorem~\ref{thm:nonperiodic} $i)$ resp.\ $ii)$, $s_n^P(z)$ converges to $s^{\sigma}(z)$ in probability resp.\ almost surely, where $s^{\sigma}$ denotes the Stieltjes transform of the semicircle distribution.
\end{proof}

\begin{corollary}\label{cor:applications}
Let $(a_n)_n$ be a Curie-Weiss ensemble with inverse temperature $\beta\in(0,1]$ (cf.\ Section~\ref{sec:curieweiss}) or an approximately uncorrelated Gaussian ensemble (cf.\ Section~\ref{sec:corgauss}). Let $h=(h_n)_n$ be a sequence of $n$-halfwidths with $h_n\to \infty$ and $\lim_n h_n/n \in\{0,1\}$. Let $(X^{NP}_n)_n$ be the non-periodic random band matrices which are based on $(a_n)_n$ with halfwidth $h$.
Then the following statements hold:
\begin{enumerate}[i)]
\item The semicircle law holds for $(X^{NP}_n)_n$ in probability.
\item If $\frac{1}{h_n^p}$ is summable over $n$ for some $p\in\N$, then the semicircle law holds almost surely for $(X^{NP}_n)_n$. 
\end{enumerate}
\end{corollary}
\begin{proof}
This is a direct consequence Theorem~\ref{thm:nonperiodic} $i)$ and $ii)$, the proof of Corollary~\ref{cor:CurieWeiss} and Lemma~\ref{lem:AUGaussian}.
\end{proof}

\section{Weighted Ensembles with Two Layers of Correlation}
\label{sec:weightedmatrix}
\subsection{Setup and Results}
We consider the following setup: We assume that $(a_n)_n$ is a sequence of real-symmetric $n\times n$ random matrices. We do not assume the families of random variables $(a_n(i,j))_{1\leq i\leq j\leq n}$ to be independent, nor do we require them to be standardized. Rather, we allow arbitrarily high correlation (even equality) of random variables which belong to certain subfamilies, and that random variables from different subfamilies are approximately uncorrelated. To make this precise, we assume that for all $n\in\N$, $\sim_n$ is an equivalence relation on $\oneto{n}^2$ which satisfies the following conditions (when $n$ can be derived from the context, we write $\sim$ instead $\sim_n$): There exists a $B\in\N$ independent of $n$ such that
\begin{align*}
(E1) &\quad \max_{p\in\oneto{n}} \#\left\{(q,r,s)\in\oneto{n}^3 ~|~ (p,q)\sim (r,s)\right\} = o(n^2),\\
(E2) &\quad \max_{p,q,r\in\oneto{n}} \# \left\{s\in\oneto{n} ~|~ (p,q)\sim (r,s)\right\} \leq B,\\
(E3) & \quad \#\left\{(p,q,r)\in\oneto{n}^3 ~|~ (p,q)\sim (q,r)\wedge r\neq p\right\} = o(n^2).
\end{align*}
These are exactly the same conditions as (C1), (C2) and (C3) in \cite{Schenker:Schulz-Baldes:2005}. For some of our results, we also require the following conditions: There exists a fixed $\delta>0$ independent of $n$ such that
\begin{align*}
(E1') &\quad \max_p \#\left\{(q,r,s)\in\oneto{n}^3 ~|~ (p,q)\sim (r,s)\right\} = O(n^{2-\delta}),\\
(E3') & \quad \#\left\{(p,q,r)\in\oneto{n}^3 ~|~ (p,q)\sim (q,r)\wedge r\neq p\right\} = O(n^{2-\delta}).
\end{align*}
We observe that $(E1')$ and $(E3')$ are slightly stronger than their counterparts $(E1)$ and $(E3)$. The stronger conditions will be used to derive almost sure convergence results.

In the setup of \cite{Schenker:Schulz-Baldes:2005}, the entries of $a_n$ were assumed to be standardized, have uniformly bounded absolute moments of all orders, and that $\sim$ was required to satisfy $(E1)$, $(E2)$, $(E3)$. Further, it was assumed that the families
\begin{equation}
\label{eq:families}
(a_n(P))_{P\in M},\qquad M\in\nicefrac{\oneto{n}^2}{\sim}
\end{equation}
be independent while no independence requirement was made for members of the same equivalence class (for example, they could be all the same random variable). Due to independence between different equivalence classes it was also necessary to assume that $(p,q)\sim(q,p)$ for all $(p,q)\in\oneto{n}^2$, since the matrices $a_n$ are symmetric.

In our setup, we also assume that $\sim$ satisfies $(E1)$, $(E2)$ and $(E3)$ (and for some results $(E1')$, $(E2)$ and $(E3')$) and that $(p,q)\sim(q,p)$ for all $(p,q)\in\oneto{n}^2$. However, we drop the standardization requirement and the requirement of independence between equivalence classes. We instead require entries from different equivalence classes to be approximately uncorrelated in the sense of Section~\ref{sec:AUsetup}: Let $\ell,s\geq 0$ be arbitrary, $P_1,\ldots,P_{s},Q_1,\ldots,Q_{\ell}\in\oneto{n}^2$ be distinct index pairs, where $P_1,\ldots,P_s$ stem from distinct $\sim_n$-equivalence classes, and let $\delta_{1}\ldots,\delta_{\ell}\in\N$, then

\begin{align}
&\quad \bigabs{\E a_n(P_1)\cdots a_n(P_s) a_n(Q_{1})^{\delta_1}\cdots a_n(Q_{\ell})^{\delta_{\ell}}} \ \leq\ \frac{C(s + \delta_1+\ldots +\delta_{\ell})}{n^{\frac{s}{2}}},\label{eq:A1}\tag{A1}\\
&\quad \bigabs{\E a_n(P_1)^2\cdots a_n(P_s)^2-1}\ \leq\ C^{(s)}_n,\label{eq:A2}\tag{A2} 
\end{align}
where $C(s)$ resp.\ $(C^{(s)}_n)_n$ are constants resp.\ sequences for all $s\in\N$, where for all $s$, $C^{(s)}_n\to 0$ as $n\to\infty$. Note that \eqref{eq:A1} implies that all entries in $a_n$ have uniformly bounded (absolute) moments of all orders and \eqref{eq:A1} and \eqref{eq:A2} together imply that all entries in $a_n$ are asymptotically standardized. Clearly, the case of standardized entries with independent families \eqref{eq:families} -- the setup of \cite{Schenker:Schulz-Baldes:2005} -- is included in above setup as a special case. 
\begin{definition}
\label{def:SSB-HKW}
A triangular scheme $(a_n)_n$ satisfying \eqref{eq:A1} and \eqref{eq:A2}, where equivalence within $\oneto{n}^2$ is governed by a relation $\sim_n$ satisfying $(E1)$, $(E2)$, and $(E3)$ (or $(E1')$, $(E2)$ and $(E3')$ when this is explicitly stated), will be called \emph{SSB-HKW correlated}.
\end{definition}

Next, we assume that $w: [0,1]\to\R$ is a Riemann integrable weight function. In particular, $w$ is bounded, $\abs{w}\leq W$ for some $W\in\R_+$. We consider weighted matrices of the form
\begin{equation}
\label{eq:matrixform}
X_n \defeq \frac{1}{\sqrt{n}}\left[w(\abs{i-j})a_n(i,j)\right]_{1\leq i,j \leq n}.
\end{equation}
We observe that within the diagonals of $X_n$, the same weight is applied. This could be generalized, as has been done in \cite{Zhu:2020} via graphon theory. However, our study is motivated primarily by band matrices which makes our modeling natural.

Under the setup we just described, we now formulate our main theorem, which summarizes all the results of this second part of the paper:

\begin{theorem}
\label{thm:weightedmatrix}
Let $\sigma_n$ be the ESD of $X_n$ as in \eqref{eq:matrixform}, where $(a_n)_n$ is SSB-HKW correlated. Then $\sigma_n$ converges weakly in probability to a symmetric and compactly supported probability measure $\mu$ on $(\R,\Bcal)$, which is uniquely determined by its moments
\[
\forall\, k\in\N: \integrala{\mu}{x^k} = \sum_{\pi\in\Ncal\Pcal\Pcal(k)} J_{w}(\pi),
\]
where the constants $J_{w}(\pi)$ depend on the weight function $w$ and the partition $\pi$, and can be calculated recursively as described in Lemma~\ref{lem:Jwpi} below. Further, set
\[
\forall\, x\in[0,1] : \quad \varphi(x) \defeq \int_0^1 w^2(\abs{x-y})\de y
\]
and  $\varphi_0\defeq \int_0^1\varphi(x)\de x$. Then the limiting variance is given by $\varphi_0$, i.e.\ $\integrala{\mu}{x^2}=\varphi_0$, and can be calculated by
\[
\varphi_0 = 2\int_0^1 (1-x)w^2(x)\de x.
\]
In particular, $\varphi_0 =0$ if and only if $w = 0$ on $[0,1]$ $\lebesgue$-almost surely.	 In the case that $\varphi_0 >0$, then the following statements are equivalent:
\begin{enumerate}[a)]
\item The semicircle law holds for $\frac{1}{\sqrt{\varphi_0}}X_n$ in probability.
\item $\varphi$ is constant, in particular, $\varphi\equiv\varphi_0$.
\item $\varphi \equiv \int_0^1 w^2(x)\de x$,
\item $w^2$ is $\lebesgue$-a.s.\ symmetric around $1/2$, i.e.\ $w^2(x)=w^2(1-x)$ for $\lebesgue$-a.a.\ $x\in[0,1]$.
\end{enumerate}
Further, if $(E1)$ and $(E3)$ are replaced by $(E1')$ and $(E3')$, all statements about weak convergence in probability can be replaced by weak convergence almost surely.
\end{theorem}
\subsection{Examples}
In this subsection we study various types of examples and counterexamples which illustrate the reach of Theorem~\ref{thm:weightedmatrix}. 
\subsubsection{Repetitions of approximately uncorrelated entries}
Assume that $\sim_n$ and $w$ are as above and that $a_n'$ is an approximately uncorrelated triangular scheme, for example a Curie-Weiss($\beta$)-ensemble as in Section~\ref{sec:curieweiss} or an approximately uncorrelated Gaussian ensemble as in Section~\ref{sec:corgauss}. Let $\ell\in\N$ be the number of equivalence classes induced by $\sim_n$, and let $P_1,\ldots,P_{\ell}\in\oneto{n}^2$ be representatives. Then for all $P\in\oneto{n}^2$ we set $a_n(P)\defeq a_n'(P_i)$, where $i\in\oneto{\ell}$ is the unique index with $P\sim_n P_i$. Then $(a_n)_n$ is an SSB-HKW correlated ensemble as in Definition~\ref{def:SSB-HKW}. 

\subsubsection{The $k$-band model} The $k$-band model generalizes periodic and non-periodic band matrices to multiple bands. We assume that $a_n$ is an SSB-HKW correlated ensemble and that $w=\one_{I_1}+\ldots+\one_{I_k}$, where $(I_j)_{j\in\oneto{k}}$ are non-degenerate intervals in $[0,1]$ that satisfy $I_1 < I_2 <\ldots < I_k$ (where the inequalities are meant element-wise) and we assume that there is an $\epsilon>0$ such that for all $i\neq j\in\oneto{k}$, $\dist(I_i,I_j)\geq \epsilon$, where $\dist(I_i,I_j)\defeq\inf\{\abs{x-y}\,|\, x\in I_j,\,y\in I_j\}$. In this setup, Theorem~\ref{thm:weightedmatrix} ensures convergence of the ESDs of the random matrices $X_n$ based on $a_n$ as in \eqref{eq:matrixform}.
By the symmetry characterization in Theorem~\ref{thm:weightedmatrix} it is immediately clear -- without further calculations -- that the SCL holds iff the intervals are a.s.\ symmetric around $1/2$, which means that for $\lebesgue$-almost all for all $x\in[0,1]$: $x\in I_{i}$ for some $i\in\oneto{k}$ iff $(1-x)\in I_{j}$ for some $j\in\oneto{k}$. As an example, if
\begin{equation}
\label{eq:examplew}
w = \one_{\left[\frac{1}{10},\frac{2}{10}\right]} + \one_{\left[\frac{3}{10},\frac{4}{10}\right]} + \one_{\left[\frac{6}{10},\frac{7}{10}\right]} + \one_{\left[\frac{8}{10},\frac{9}{10}\right]},
\end{equation}
then $w$ is symmetric around $1/2$, so the SCL holds for the associated $X_n$. But if exactly one of the indicators in \eqref{eq:examplew} is dropped, the SCL will fail to hold for $X_n$.

An important application of this observation is for periodic and non-periodic band matrices as in Definition~\ref{def:bandmatrices} with a halfwidth $h_n$ which grows proportionally with $n$. So we assume $h_n/n \to\rho$ with $\rho\in(0,1)$. If $X^{NP}_n$ is based on an approximately uncorrelated triangular scheme $a_n$ with halfwidth $h_n$, then $X^{NP}_n$ and $X_n$ share the lame limiting spectral distribution (LSD) if $X_n$ is as in \eqref{eq:matrixform}, where $a_n$ is the same approximately uncorrelated triangular scheme, $\sim_n$ is the trivial equivalence relation $(a,b)\sim_n(p,q)$ iff $(a,b)=(p,q)$ or $(a,b)=(q,p)$, and where $w=\one_{[0,\rho]}$. This is easily seen, since by Definition~\ref{def:bandmatrices}, 
\[
X_n^{NP}(a,b)\neq 0 \Leftrightarrow \abs{a-b}\leq h_n-1 \Leftrightarrow \frac{\abs{a-b}}{n} \leq \frac{h_n-1}{n} \to \rho,
\]
and $X_n(a,b)\neq 0$ iff $\abs{a-b}/n \leq \rho$ by \eqref{eq:matrixform}. 
Therefore, the limiting spectral distributions of $X_n^{NP}$ and $X_n$ are the same, as can be seen by using standard perturbation arguments (e.g.\ Hoffman-Wielandt inequality). With the same argument we obtain that the LSDs of $X_n^{P}$ and $X'_n$ are the same, where $X'_n$ is constructed the same as $X_n$, but we change the weight $w$ to the weight $w'=\one_{[0,\rho]\cup[1-\rho,1]}$. Since $w'$ is symmetric around $1/2$, but $w$ is not, the SCL holds for $X_n^{P}$, but not for $X_n^{NP}$.
Of course, for independent entries, these statements about $X^{P}_n$ and $X_n^{NP}$ are well-known \cite{Bogachev:Molchanov:Pastur:1991}.

\subsubsection{Block matrices} 
We consider multiple types of block matrices -- all with finitely many different blocks -- which can be treated with Theorem~\ref{thm:weightedmatrix}, and we will also construct a type of block matrix which cannot be treated with Theorem~\ref{thm:weightedmatrix}, exemplifying the boundaries of this theorem.\newline
\emph{Toeplitz block matrices with $k$ diagonals.} A Toeplitz block matrix has the same blocks along diagonals, thus is of the form
\[
T^{(k)}_n = \frac{1}{\sqrt{kn}}
\begin{pmatrix}
A^{(1)}_n & A^{(2)}_n & A^{(3)}_n & \cdots & A^{(k)}_n\\
A^{(2)T}_n & A^{(1)}_n & A^{(2)}_n & \cdots & A^{(k-1)}_n\\
A^{(3)T}_n & A^{(2)T}_n & A^{(1)}_n & \cdots & A^{(k-2)}_n\\
\vdots & \vdots & \vdots & \ddots & \vdots\\
A^{(k)T}_n & A^{(k-1)T}_n & A^{(k-2)T}_n &\cdots & A^{(1)}_n
\end{pmatrix},
\]
where the matrix $A_n^{(1)}$ is symmetric and $T$ indicates the transpose of a given matrix. 
We assume that the family
\[
\{A^{(1)}_n(i,j)\,|\, 1\leq i\leq j\leq n\} \cup \{A^{(\ell)}(P)\,|\, P\in\oneto{n}^2, \ell\in\{2,\ldots,k\}\}
\]
is approximately uncorrelated. Then the semicircle holds almost surely for $T^{(k)}_n$ by Theorem~\ref{thm:weightedmatrix}. To see this, we need to verify that $\sim_{kn}$ induced by the structure of $T^{(k)}_n$ satisfies the conditions $(E1')$, $(E2)$ and $(E3')$. For $(E1')$, fix a $p\in\oneto{kn}$, then for any $q\in\oneto{kn}$ there are at most $4k$ pairs $(r,s)$ in $\oneto{n}^2$ equivalent to $(p,q)$. Since  $4k^2n=O((kn)^{1.9})$, $(E1')$ is satisfied. For $(E2)$, fix $p,q,r\in\oneto{kn}$. Then are at most two $s\in\oneto{kn}$ so that $(p,q)$ is equivalent to $(r,s)$. Hence, $(E2)$ is satisfied with $B=2$. For $(E3')$, we note that if $(p,q)$ is fixed, then there is no $r\neq q$ so that $(p,q)$ is equivalent to $(q,r)$, so that $(E3')$ is satisfied.

\noindent
\emph{Hankel block matrices with $k$ skew diagonals.} For $k$ odd, a Hankel block matrix has the same blocks along its skew diagonals, thus has the form
\[
H^{(k)}_n = \frac{1}{\sqrt{\frac{k+1}{2}n}}
\begin{pmatrix}
A^{(1)}_n & A^{(2)}_n & A^{(3)}_n  & \cdots & A^{((k+1)/2)}_n\\
A^{(2)T}_n & A^{(3)}_n &  A^{(4)}_n & \cdots & A^{((k+3)/2)}_n\\
A^{(3)T}_n & A^{(4)T}_n & A^{(5)}_n & \cdots & A^{((k+5)/2)}_n\\
\vdots & \vdots & \vdots & \ddots & \vdots\\
A^{((k+1)/2)T}_n & A^{((k+3)/2)T}_n & A^{((k+5)/2)T}_n & \cdots & A^{(k)}_n
\end{pmatrix},
\]
where the matrices $A^{(1)}_n,A^{(3)}_n,\ldots,A^{(k)}_n$ are symmetric. We assume that the family
\[
\{A^{(\ell)}_n(i,j)\,|\, 1\leq i\leq j\leq n, \ell\in\{1,3,\ldots,k\}\} \cup \{A^{(\ell)}(P)\,|\, P\in\oneto{n}^2, \ell\in\{2,4,\ldots,k-1\}\}
\]
is approximately uncorrelated. Then the semicircle law holds almost surely for $H^{(k)}_n$ by Theorem~\ref{thm:weightedmatrix}. To see this, we again need to verify that $\sim_{kn}$ induced by the structure of $T^{(k)}_n$ satisfies the conditions $(E1')$, $(E2)$ and $(E3')$. For $(E1')$, fix a $p\in\oneto{kn}$, then for any $q\in\oneto{kn}$ there are at most $4k$ pairs $(r,s)$ in $\oneto{n}^2$ equivalent to $(p,q)$. Since  $4k^2n=O(((k+1)/2)n)^{1.9})$, $(E1')$ is satisfied. For $(E2)$, fix $p,q,r\in\oneto{kn}$. Then there are at most two $s\in\oneto{kn}$ so that $(p,q)$ is equivalent to $(r,s)$. Hence, $(E2)$ is satisfied with $B=2$. For $(E3')$, we note that if $(p,q)$ is fixed, then there is no $r\neq q$ so that $(p,q)$ is equivalent to $(q,r)$, so that $(E3')$ is satisfied.

\noindent
\emph{Homogeneous block model.} We assume that the block matrix is made up of $k\times k$ blocks, where the diagonal blocks are the same, and also the non-diagonal blocks are the same (up to transposition).
\[
M^{(k)}_n \defeq \frac{1}{\sqrt{kn}} 
\begin{pmatrix}
A_n & B_n & B_n & \cdots & B_n\\
B_n^T & A_n & B_n & \cdots & B_n\\
B_n^T & B_n^T & A_n & \cdots & B_n\\
\vdots & \vdots & \vdots & \ddots & \vdots\\
B_n^T & B_n^T & B_n^T &\cdots & A_n
\end{pmatrix},
\]
where $A_n$ is symmetric. We assume that the family
\begin{equation}
\label{eq:Mfamily}
\{A_n(i,j)\,|\, 1\leq i\leq j\leq n\}\} \cup \{B_n(P)\,|\, P\in\oneto{n}^2\}
\end{equation}
is approximately uncorrelated. We check if $\sim_{kn}$ induced by the structure of $M^{(k)}_n$ satisfies the conditions $(E1)$, $(E2)$ and $(E3)$. For $(E1)$, fix a $p\in\oneto{kn}$, then for any $q\in\oneto{kn}$ there are at most $4k^2$ pairs $(r,s)$ in $\oneto{n}^2$ equivalent to $(p,q)$. Since  $4k^2n=O((kn)^{1.9})$, even $(E1')$ is satisfied. For $(E2)$, fix $p,q,r\in\oneto{kn}$. Then are at most $2(k-1)$ elements $s\in\oneto{kn}$ so that $(p,q)$ is equivalent to $(r,s)$. Hence, $(E2)$ is satisfied with $B=2(k-1)$. For $(E3')$, we note that if $(p,q)$ is fixed such that it falls within a block $B_n$ or $B_n^T$ (for which we have $(nk)^2-kn^2$ choices), then for $k\geq 3$ there are exactly $k-2$ elements $r\neq q$ so that $(p,q)$ is equivalent to $(q,r)$, so that $(E3)$ is not satisfied for $k\geq 3$, since then $((nk)^2-kn^2)(k-2)\neq o((kn)^2)$. So Theorem~\ref{thm:weightedmatrix} is not applicable for $k\geq 3$. $(E3')$ is, however, satisfied for $k\in\{1,2\}$, so that the almost sure semicircle holds in that case. A simulation of $M_n^{(k)}$ for $k=2,3,4$, $n=1000$ and an i.i.d.\ Rademacher family \eqref{eq:Mfamily} indicates that indeed, the SCL is likely to fail for $k\geq 3$, see Figure~\ref{fig:HomBand}.

\begin{figure}
    {\centering
    \begin{minipage}[t]{0.3\linewidth}
        \centering
        \includegraphics[clip=true, trim=1cm 0cm 0cm 0cm, width=\linewidth]{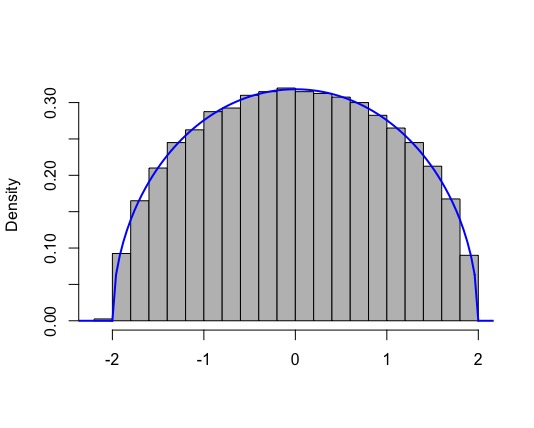}
    \end{minipage}
    \hfill
    \begin{minipage}[t]{0.3\linewidth}
        \centering
        \includegraphics[clip=true, trim=1cm 0cm 0cm 0cm, width=\linewidth]{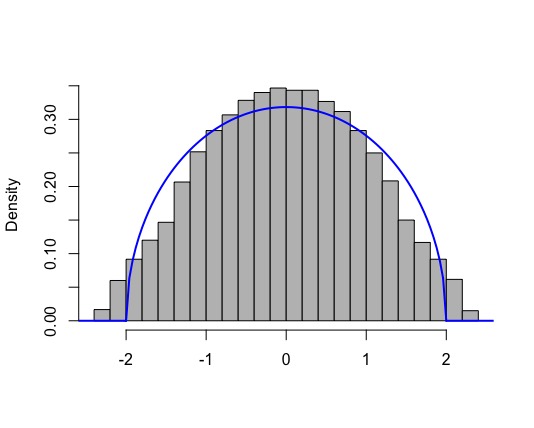}
    \end{minipage}
    \hfill
    \begin{minipage}[t]{0.3\linewidth}
        \centering
        \includegraphics[clip=true, trim=1cm 0cm 0cm 0cm, width=\linewidth]{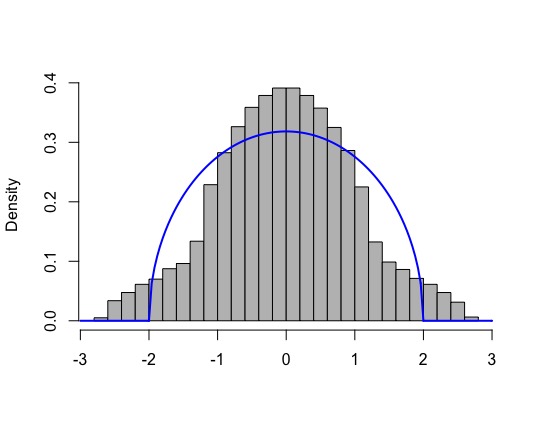}
    \end{minipage}}
    \caption{Histograms (in grey) of simulated spectra of $M_n^{(k)}$ for $n=1000$ and $k=2,3,4$ and i.i.d.\ Rademacher family \eqref{eq:Mfamily}, density (in blue) of the semicircle distribution.}
\label{fig:HomBand}
\end{figure}

\subsection{Proof of Theorem~\ref{thm:weightedmatrix}}
In Section~\ref{sec:momentconvergence} we derive the convergence of the moments of the ESDs studied in Theorem~\ref{thm:weightedmatrix}. In Section~\ref{sec:limitingmoments} we analyze the limiting moments, argue that these uniquely determine the limiting distribution which has to have compact support and be symmetric, and derive the weak convergence statements of Theorem~\ref{thm:weightedmatrix}. Section~\ref{sec:characterization} is devoted to the last statement of Theorem~\ref{thm:weightedmatrix}, the characterization of the weak limit being the semicircle distribution.

\subsubsection{Convergence of Moments}
\label{sec:momentconvergence}
 In this subsection, we will prove the weak convergence statements in Theorem~\ref{thm:weightedmatrix}. To this end, we need to develop some notation and combinatorics.
For every tuple $\ubar{t}\in\oneto{n}^k$, denote by $\pi_{
\ubar{t}}$ the partition on $\oneto{k}$ with
\[
i\sim_{\pi_{\ubar{t}}} j \ :\Leftrightarrow \ e_i \sim_n e_j \quad \left(\Leftrightarrow (t_i,t_{i+1}) \sim_n (t_j,t_{j+1})\right),
\]
where $k+1$ is identified with $1$ and we recall the interpretation of $\ubar{t}$ as a graph as in the description right below \eqref{eq:sum2}, so that $e_1,\ldots,e_k$ are the edges of $\ubar{t}$.
Also, note that $\pi_{\ubar{t}}$ depends on $n$, but we will suppress this dependence in the notation. Now if $\pi$ is any partition on $\oneto{k}$, then define
\[
T_{n,k}(\pi) \defeq \{\ubar{t}\in\oneto{n}^k\,|\, \pi_{\ubar{t}}=\pi\}.
\]
Then 
\[
\oneto{n}^k = \dot{\cup}_{\pi\in\Pcal(k)} T_{n,k}(\pi),
\]
where $\Pcal(k)$ denote the set of partitions of the set $\oneto{k}$.
Further, define for all $\ubar{t}\in\oneto{n}^k$:
\begin{align*}
a_n(\ubar{t}) &\defeq a_n(t_1,t_2)a_n(t_2,t_3)\cdots a_n(t_k,t_1) = \prod_{j=1}^k a_n(t_j,t_{j+1}),\\
w(\ubar{t}) &\defeq w\left(\frac{\abs{t_1-t_2}}{n}\right)w\left(\frac{\abs{t_2-t_3}}{n}\right)\cdots w\left(\frac{\abs{t_k-t_1}}{n}\right) = \prod_{j=1}^kw\left(\frac{\abs{t_j-t_{j+1}}}{n}\right).
\end{align*}
Then
\begin{equation}
\label{eq:momenttosum1}
\integrala{\sigma_n}{x^k} = \frac{1}{n^{\frac{k}{2}+1}} \sum_{\ubar{t}\in\oneto{n}^k}w(\ubar{t})a_n(\ubar{t}) = \frac{1}{n^{\frac{k}{2}+1}}\sum_{\pi\in\Pcal(k)}\sum_{\ubar{t}\in T_{n,k}(\pi)}w(\ubar{t})a_n(\ubar{t}).	
\end{equation}

\begin{lemma}
\label{lem:tuplescount}	
Let $\pi\in\Pcal(k)$ be a partition with $\#\pi = r$ and $n\geq B$, then we have
\[
\# T_{n,k}(\pi) \leq n^{r+1}B^{k-r-1}.
\]
Further, if $\pi$ contains a singleton, we have 
\[
\# T_{n,k}(\pi) \leq  n^{r}B^{k-r}.
\]
\end{lemma}
\begin{proof}
The first statement was proved in \cite[p.4]{Schenker:Schulz-Baldes:2005}. The second statement is clear for $k=1$. If $k\geq 2$, write $\pi=\{B_1,\ldots,B_r\}$ in increasing order, that is, $1\in B_1$, $\min\{1,\ldots,k\}\backslash B_1 \in B_2$, and so forth. Now for some $s\in\oneto{r}$, $\# B_s =1$, so $B_s=\{i^*\}$ for some $i^*\in\oneto{k}$. We begin the construction of $\ubar{t}$ at the edge $e_{i^*+1} = (t_{i^*+1},t_{i^*+2})$, for which we have $n^2 $ choices. For each $\ell\in\{i^*+2,\ldots,i^*-1\}$ we have to choose the destination of edge $e_{\ell}$, which is $t_{\ell+1}$ (note that the destination of edge $e_{i^*}$ has already been picked). Each time $\ell$ enters a block of $\pi$ which has not been visited before, we have at most $n$ choices for $t_{\ell+1}$, and this happens exactly $r-2$ times. If $\ell$ does not enter a new block, then $e_{\ell}$ is $\pi$-equivalent to some $e_{\ell'}$ with $\ell'\in\{i^*+1,\ldots,\ell-1\}$, leaving at most $B$ choices for $t_{\ell+1}$ due to condition $(E2)$, which happens exactly $k-2-(r-2)$ times, yielding at most $n^2\cdot n^{r-2}B^{k-r}$ choices.
\end{proof}
Next, we analyze the following situation: Assume $\ubar{t}$ and $\ubar{t}'\in T_{n,k}(\pi)$ are two tuples. Then we know that \emph{separately}, the edges $e_1,\ldots, e_k$ and $e'_1,\ldots,e'_k$ are compatible with $\pi$. But what can be said about the partition $\hat{\pi}\in\Pcal(2k)$ that describes the $\sim_n$ relations of the entirety $e_1,\ldots,e_k,e'_1,\ldots,e'_k$ (where $e_k=(t_k,t_1)$ and $e'_k =(t'_k,t'_1)$)? Surely, restricting $\hat{\pi}$ to the lower and upper half of $\oneto{2k}$  by setting 
\[
\hat{\pi}_{\leq k} \defeq \{B\cap\{1,\ldots,k\},B\in\hat{\pi}\} \qquad \text{and}\qquad \hat{\pi}_{> k} \defeq \{B\cap\{k+1,\ldots,2k\},B\in\hat{\pi}\},
\]
then $\hat{\pi}_{\leq k} = \pi = \hat{\pi}_{> k}$, where $\{k+1,\ldots,2k\}$ is identified with $\oneto{k}$. In this situation, that is, we are given a $\pi\in\Pcal(k)$ and a $\hat{\pi}\in\Pcal(2k)$ with $\hat{\pi}_{\leq k} = \pi = \hat{\pi}_{> k}$, we say that $\hat{\pi}\in\Pcal(2k)$ is a \emph{square partition} of $\pi$. Note that in general, there are many possible square partitions to a given partition. We say that a square partition $\hat{\pi}$ is \emph{traversing}, if there is a block $B\in\hat{\pi}$ with $B\cap\{1,\ldots,k\}\neq\emptyset$ and $B\cap\{k+1,\ldots,2k\}\neq\emptyset$. Otherwise, we call $\hat{\pi}$ \emph{non-traversing}. We are now interested in bounds of the cardinality of the following sets:
\[
\hat{T}_{n,2k}(\hat{\pi})\defeq\left\{(\ubar{t},\ubar{t}') \in (T_{n,k}(\pi))^2, (e_1,\ldots,e_k,e'_1,\ldots,e'_k)\sim_n\hat{\pi} \right\}.
\]
Our bounds will depend on the number of unifications of blocks in $\hat{\pi}_{\leq k}$ with blocks in $\hat{\pi}_{> k}$ carried out in $\hat{\pi}$. Note that there may be at most $\#\pi$ unifications.

\begin{lemma}
\label{lem:tuplescount2}	
Let $\hat{\pi}\in\Pcal(2k)$ be a square partition of $\pi\in\Pcal(k)$ with $\#\pi=r$. 
\begin{enumerate}
\item Let $\hat{\pi}$ be non-traversing. Then $\#\hat{\pi}=2r$, and for all $n\geq B$
\[
\# \hat{T}_{n,2k}(\hat{\pi}) \leq n^{2r+2}B^{2k-2r-2}.
\]
Further, if $\hat{\pi}$ contains a singleton, we have 
\[
\# \hat{T}_{n,2k}(\hat{\pi}) \leq  n^{2r}B^{2k-2r}.
\]
\item If $k\geq 2$, $\hat{\pi}$ is traversing, and $1\leq u\leq r$ is the number of unifications, then $\#\hat{\pi}=2r-u $ and for $n\geq B$ we have
\[
\# \hat{T}_{n,2k}(\hat{\pi}) \leq n^{2r-u}o(n^2)B^{2k-2r+u-3}.
\]
Further, if $s_1(\pi)\geq 2$, that is, $\pi$ contains at least two singletons, it holds 
\[
\# \hat{T}_{n,2k}(\hat{\pi}) \leq n^{2r-u-1}o(n^2)B^{2k-2r+u-2}.
\]
In addition, the terms $o(n^2)$ may be replaced by $O(n^{2-\delta})$ if $(E1')$ is assumed instead of $(E1)$.
\end{enumerate}
\end{lemma}
\begin{proof}
The statements in $(1)$ follow directly from Lemma~\ref{lem:tuplescount} and the inequality
\[
\# \hat{T}_{n,2k}(\hat{\pi}) \leq \# T_{n,k}(\pi)\cdot\# T_{n,k}(\pi),
\]
which is a trivial upper bound, since for any pair of tuples $(\ubar{t},\ubar{t}')\in \hat{T}_{n,2k}(\hat{\pi})$ we have $\ubar{t}$, $\ubar{t}'\in T_{n,k}(\pi)$.
For the statements in $(2)$ we analyze how many possibilities we have to construct such a tuple pair. For the first statement we begin with the smallest $\ell\in\oneto{k}$ which is $\hat{\pi}$-equivalent to some $\ell'\in\{k+1,\ldots,2k\}$. To fix $t_{\ell},t_{\ell+1},t_{\ell'}',t_{\ell'+1}'$, we have $n\cdot o(n^2)$ possibilities by condition $(E3)$. Continuing through $\ubar{t}$, we have at most $n^{r-1}B^{k-r-1}$ possibilities. When continuing through the edges of $\ubar{t}'$, there are $r-1$ new blocks to be entered, $u-1$ of which will be connected with some block in $\hat{\pi}_{\leq k}$, which will then admit $B^{u-1}$ choices by condition $(E2)$. The remaining $r-u$ independent new blocks will admit at most $n^{r-u}$ choices. The remaining edges will admit at most $B^{k-r-1}$ choices by condition $(E2)$. This yields at most
\[
n\cdot o(n^2)\cdot n^{r-1}B^{k-r-1} \cdot B^{u-1} \cdot n^{r-u} \cdot B^{k-r-1} = n^{2r-u}\cdot o(n^2)\cdot B^{2k-2r+u-3}
\]
choices to construct a pair of tuples $(\ubar{t},\ubar{t}')\in \hat{T}_{n,2k}(\hat{\pi})$.

If $\pi$ contains at least two singletons, we proceed similarly. Starting at the first $\ell\in\oneto{k}$ which is $\hat{\pi}$-equivalent to some $\ell'\in\{k+1,\ldots,2k\}$, we again fix $t_{\ell},t_{\ell+1},t_{\ell'}',t_{\ell'+1}'$ for which we have at most $n\cdot o(n^2)$ possibilities by condition $(E3)$. To complete the tuple $\ubar{t}$ we pick an $m\in\oneto{k}\backslash\{\ell\}$ such that $e_m$ is a single edge and continue exactly as in the case above except that we do not cycicly walk through the remaining edges of $\ubar{t}$, but from $e_{\ell}$ forward until reaching the node $t_m$, and then from $e_{\ell}$ backwards until reaching the node $t_{m+1}$. We have then completely determined the tuple while only entering $r-2$ new blocks, leading to at most $n^{r-2}B^{k-r}$ possibilities. To complete $\ubar{t}'$, we proceed exactly as in the proof of the first statement in $(2)$, yielding a total of at most
\[
n\cdot o(n^2)\cdot n^{r-2}B^{k-r} \cdot B^{u-1} \cdot n^{r-u} \cdot B^{k-r-1} = n^{2r-u-1}\cdot o(n^2)\cdot B^{2k-2r+u-2}
\]
possibilities to construct the pair $(\ubar{t},\ubar{t}')\in \hat{T}_{n,2k}(\hat{\pi})$.
\end{proof}

We proceed to analyze the sum in \eqref{eq:momenttosum1}. Note that for each summand,
\begin{equation}
\label{eq:universalbound}
\abs{\E w(\ubar{t})a_n(\ubar{t})} \leq W^kC(k),	
\end{equation}
which we call \emph{universal bound}. It will be used frequently throughout the text. For a partition $\pi\in\Pcal(k)$ and $\ell\in\oneto{k}$, denote by $s_{\ell}(\pi)$ the number of blocks of size $\ell$ in $\pi$, so 
\[
s_{\ell}(\pi)\defeq \#\{B\in\pi\,|\, \abs{B}=\ell\},
\]
so that
\begin{equation}
\label{eq:partitionsum}
k = \sum_{\ell=1}^k s_{\ell}(\pi)\cdot\ell \qquad \text{and} \qquad \#\pi =\sum_{\ell=1}^k s_{\ell}(\pi).
\end{equation}
The equalities in \eqref{eq:partitionsum} yield some useful inequalities. For example, for any $\pi\in\Pcal(k)$,
\begin{equation}
\label{eq:singlesboundblocks}	
\#\pi\leq s_1(\pi) + \frac{k-s_1(\pi)}{2}.
\end{equation}

We will next analyze the sum in \eqref{eq:momenttosum1}. We would like to see this $k$-th empirical moment converge to limit moments in expectation, in probability or almost surely.

We now proceed to analyze the convergence in \eqref{eq:momenttosum1}.

\begin{lemma}
\label{lem:onlykhalfblocks}
Let $\pi\in\Pcal(k)$ with $\#\pi=:r\neq k/2$, then
\begin{equation}
\label{eq:rnotkhalf}
\frac{1}{n^{\frac{k}{2}+1}}\sum_{\ubar{t}\in T_{n,k}(\pi)}w(\ubar{t})a_n(\ubar{t}) \quad \xrightarrow[n\to\infty]{}\quad 0
\end{equation}
in expectation and in probability, and the statement holds also almost surely if $(E1')$ is assumed instead of $(E1)$.
\end{lemma}
\begin{proof}
For $k=1$ the term in question is
\[
\frac{1}{n^{\frac{3}{2}}}\sum_{t\in\oneto{n}}w(0)a_n(t,t)
\]
which converges to zero in expectation and almost surely by Lemma~\ref{lem:breakaway} b). So for the remainder of the proof we assume that $k\geq 2$.

If $r<k/2$ we obtain with Lemma~\ref{lem:tuplescount} that $\#T_{n,k}(\pi)\leq n^{r+1}B^{k-r-1}\leq n^{k/2+1/2}B^{k/2-1/2}$, which leads to convergence in expectation and almost sure convergence to zero in \eqref{eq:rnotkhalf} by Lemma~\ref{lem:breakaway} b).

If $r>k/2$, then \eqref{eq:singlesboundblocks} yields $s_1(\pi)\geq 2r-k\geq 1$, so by \eqref{eq:A1},
\[
\abs{\E w(\ubar{t})a_n(\ubar{t})} \leq \frac{W^k C(k)}{n^{\frac{1}{2}(2r-k)}}
\]
and therefore by Lemma~\ref{lem:tuplescount},
\[
\bigabs{\frac{1}{n^{\frac{k}{2}+1}}\sum_{\ubar{t}\in T_{n,k}(\pi)}\E w(\ubar{t})a_n(\ubar{t})}\leq \frac{1}{n^{\frac{k}{2}+1}}n^rB^{k-r}\frac{W^kC(k)}{n^{r-\frac{k}{2}}} \leq \frac{W^kC(k)B^{k-r}}{n},
\]
so the sum in question converges to zero in expectation. We will now show that its variance decays to zero. Let $\hat{\pi}$ be a square partition of $\pi$, then it suffices to show
\begin{equation}
\label{eq:vartozero}	
\frac{1}{n^{k+2}}\sum_{(\ubar{t},\ubar{t}')\in \hat{T}_{n,2k}(\hat{\pi})}\abs{\E a_n(\ubar{t}) a_n(\ubar{t}')  -\E a_n(\ubar{t}) \E a_n(\ubar{t}')} \xrightarrow[n\to\infty]{} 0.
\end{equation}
If $\hat{\pi}$ is non-traversing, then $s_1(\hat{\pi}) = 2s_1(\pi)$, so  
\[
\abs{\E a_n(\ubar{t}) a_n(\ubar{t}')}\leq \frac{C(2k)}{n^{2r-k}} \quad\text{and}\quad \abs{\E a_n(\ubar{t})\E a_n(\ubar{t}')}\leq \frac{C(k)^2}{n^{2r-k}}.
\]
Further, $\#\hat{T}_{n,2k}(\hat{\pi})\leq n^{2r}B^{2k-2r}$ by Lemma~\ref{lem:tuplescount2},  so \eqref{eq:vartozero} holds summably fast. \newline
If $\hat{\pi}$ is traversing, then we distinguish two cases:\newline
\underline{Case 1: $s_1(\pi)\geq 2$.}\newline
There is a block in $\hat{\pi}$ which is the union of a block in $\hat{\pi}_{\leq k}$ and a block in $\hat{\pi}_{>k}$. Each union decreases $\#\hat{\pi}$ by $1$, and for each union at most $2$ singletons are paired. Let $u$ with $1\leq u\leq r$ be the number of unions, then $\#\hat{\pi} = 2r-u$ and $s_1(\hat{\pi})\geq \max(0,4r-2k-2u)$. Hence,
\[
\abs{\E a_n(\ubar{t}) a_n(\ubar{t}')}\leq \frac{C(2k)}{n^{\max(0,2r-k-u)}} \quad\text{and}\quad \abs{\E a_n(\ubar{t})\E a_n(\ubar{t}')}\leq \frac{C(k)^2}{n^{\max(0,2r-k-u)}}.
\]
By Lemma~\ref{lem:tuplescount2}, $\#\hat{T}_{n,2k}(\hat{\pi})\leq  n^{2r-u-1}o(n^2)B^{2k-2r+u-2}$.
Then since
\[
\frac{n^{2r-u-1}\cdot o(n^2)}{n^{k+2}\cdot n^{\max(0,2r-k-u)}} = \frac{1}{n}\cdot\frac{o(n^2)}{n^2}\cdot\frac{1}{n^{\max(0,2r-k-u)-(2r-k-u)}},
\]
we obtain \eqref{eq:vartozero}, and this convergence is summably fast if $(E3')$ is assumed instead of $(E3)$, since the term $o(n^2)$ can then be replaced by $O(n^{2-\delta})$.\newline
\underline{Case 2: $s_1(\pi)=1$.}\newline
Since $k\geq 2$ and $\pi$ has exactly one singleton block, $\pi$ contains at least one other block of higher cardinality than $1$. Setting $r\defeq\#\pi$, the upper bound $r\leq 1+(k-1)/2=(k+1)/2$ holds. Therefore, Lemma~\ref{lem:tuplescount} yields for $n\geq B$ that $\#T_{n,k}\leq n^rB^{k-r}\leq n^{(k+1)/2}B^{(k-1)/2}$. With these observations, we can show that the fourth central moment of the sum in \eqref{eq:rnotkhalf} decays summably fast, where we can drop the factor $w(\ubar{t})$ from the analysis, since it is bounded by $W^k$:
\begin{align*}
&\E\left(\frac{1}{n^{\frac{k}{2}+1}}\sum_{\ubar{t}\in T_{n,k}(\pi)}(a_n(\ubar{t}) - \E a_n(\ubar{t}))  \right)^4\\
&=\ \bigabs{\frac{1}{n^{2k+4}} \sum_{\ubar{t}^{(1)},\ldots, \ubar{t}^{(4)}\in T_{n,k}(\pi)} \E \prod_{s=1}^4 \left[a_n(\ubar{t}^{(s)})- \E a_n(\ubar{t}^{(s)}) \right] }\notag\\
&\leq \frac{1}{n^{2k+4}}\cdot n^{2(k+1)}B^{2(k-1)} \cdot \sum_{S\subseteq\{1,\ldots,4\}} C(k\#S)\cdot C(k)^{\#S^c},
\end{align*}
where $S^c \defeq \{1,\ldots,4\}\backslash S$ and we use that for any $\ubar{t}^{(1)},\ldots, \ubar{t}^{(4)}\in T_{n,k}(\pi)$ we find
\[
\bigabs{\E \prod_{s=1}^4 \left[a_n(\ubar{t}^{(s)})- \E a_n(\ubar{t}^{(s)}) \right]} = \sum_{S\subseteq\{1,\ldots,4\}} \bigabs{\E \prod_{s\in S} a_n(\ubar{t}^{(s)})}\cdot\prod_{s\notin S}\abs{\E a_n(\ubar{t}^{(s)})},
\]
and then \eqref{eq:AU1}, also setting $C(0)\defeq 1$.

\end{proof}

Lemma~\ref{lem:onlykhalfblocks} shows that an asymptotic contribution of the empirical $k$-th moment may only stem from those $\pi\in\Pcal(k)$ with $\#\pi=k/2$. In particular, odd empirical moments vanish as $n\to\infty$. Next, we show that only pair partitions need to be considered:
\begin{lemma}
\label{lem:onlypairpartitions}
If $\pi\in\Pcal(k)$ with $\#\pi=k/2$ but $\#B\neq 2$ for some $B\in\pi$, then
\[
\frac{1}{n^{\frac{k}{2}+1}}\sum_{\ubar{t}\in T_{n,k}(\pi)} w(\ubar{t})a_n(\ubar{t}) \ \xrightarrow[n\to\infty]{}\ 0
\]
in expectation and almost surely.
\end{lemma}
\begin{proof}
If 	$\#B\neq 2$ for some $B\in\pi$, then necessarily $s_1(\pi)>1$, since $\#\pi=k/2$. Thus by Lemma~\ref{lem:tuplescount}, $\#T_{n,k}(\pi)\leq n^{k/2}B^{k/2}$. The statement follows with Lemma~\ref{lem:breakaway}.
\end{proof}

We denote by $\Pcal\Pcal(k)$ the set of all pair partitions of the set $\oneto{k}$. Returning to \eqref{eq:momenttosum1}, we have shown so far that for each $\pi\in\Pcal(k)$ with $\pi\notin\Pcal\Pcal(k)$,
\[
\frac{1}{n^{\frac{k}{2}+1}}\sum_{\ubar{t}\in T_{n,k}(\pi)}w(\ubar{t})a_n(\ubar{t}) \xrightarrow[n\to\infty]{} 0 \quad \text{in exp./prob./a.s.,}
\]
where we used properties $(E1)$ and $(E3)$ for convergence in expectation and in probability and the stronger conditions $(E1')$ and $(E3')$ for convergence almost surely. We obtain
\begin{equation}
\label{eq:momenttosum2}
\integrala{\sigma_n}{x^k}\ =\ \frac{1}{n^{\frac{k}{2}+1}}\sum_{\pi\in\Pcal\Pcal(k)}\sum_{\ubar{t}\in T_{n,k}(\pi)}w(\ubar{t})a_n(\ubar{t})\ +\ R^{(1)}_n,
\end{equation}
where $R^{(1)}_n$ is a remainder term that tends to zero in expectation and in probability under $(E1)$ and $(E3)$, and also almost surely under $(E1')$ and $(E3')$. It remains to investigate the first summand on the r.h.s.\ of \eqref{eq:momenttosum2}.

We call a partition $\pi\in\Pcal\Pcal(k)$ \emph{crossing} if there are $1\leq a<b<c<d\leq k$ such that $\{a,c\}$, $\{b,d\}\in\pi$. Otherwise, $\pi$ is called \emph{non-crossing}. Denote by $\Ncal\Pcal\Pcal(k)\subseteq\Pcal\Pcal(k)$ the subset of all non-crossing pair partitions. Further, if $\pi\in\Ncal\Pcal\Pcal(k)$, define
\begin{align*}
XT_{n,k}(\pi)&\defeq\{\ubar{t}\in T_{n,k}(\pi)\ | \ \exists\, r\neq s\in\oneto{k}: r\sim_{\pi} s \wedge (t_r,t_{r+1})\neq (t_{s+1},t_{s})\},\\
FT_{n,k}(\pi)&\defeq\{\ubar{t}\in T_{n,k}(\pi)\ |\ \{r,s\}\in\pi \Rightarrow (t_r,t_{r+1})= (t_{s+1},t_{s}), \# V_{\ubar{t}}\leq k/2\},\\
BT_{n,k}(\pi)&\defeq\{\ubar{t}\in T_{n,k}(\pi)\ |\ \{r,s\}\in\pi \Rightarrow (t_r,t_{r+1})= (t_{s+1},t_{s}), \# V_{\ubar{t}}= k/2+1\}.
\end{align*}
\begin{lemma}
\label{lem:reductionSSB}
Recall the constant $B$ from condition $(E2)$. Then we find:
\begin{enumerate}[a)]
\item If $\pi\in\Pcal\Pcal(k)$ is crossing, then
\[
\#T_{n,k}(\pi)\ \leq\ n^{\frac{k}{2}-1}\cdot K(k,B) \cdot o(n^2),
\]	
where $o(n^2)$ is the term from condition $(E3)$, which may be replaced by the term $O(n^{2-\delta})$ from condition $(E3')$ if the latter condition is assumed. Further, $K(k,B)$ is a constant which depends only on $k$ and $B$.
\item If $\pi\in\Ncal\Pcal\Pcal(k)$, then
\[
\#XT_{n,k}(\pi)\ \leq\ n^{\frac{k}{2}-1}\cdot K(k,B) \cdot o(n^2),
\]
where $o(n^2)$ is the term from condition $(E3)$, which may be replaced by the term $O(n^{2-\delta})$ from condition $(E3')$ if the latter condition is assumed. Further, $K(k,B)$ is a constant which depends only on $k$ and $B$.
\item If $\pi\in\Ncal\Pcal\Pcal(k)$, then
\[
\# FT_{n,k}(\pi) \leq K(k)\cdot n^{\frac{k}{2}},
\]
where $K(k)$ is a constant depending only on $k$.
\end{enumerate}	
\end{lemma}

\begin{proof}
a) Is more detailed version of Lemma 2 in \cite{Schenker:Schulz-Baldes:2005}, page 5. Their Lemma 2 is proved with a reduction argument (their Lemma 1 on page 4). The constant $K(k,B)$ will depend on the number of reductions, which depends on the specific non-crossing $\pi$, but is upper bounded by $k/2-1$.\newline
b) This statement is a more detailed version of Lemma 3 in \cite{Schenker:Schulz-Baldes:2005}, where the reasoning for the constant $K(k,B)$ is similar to the case of a).\newline
c) We bound $\#FT_{n,k}(\pi)$ by the number of possibilities to construct a tuple with at most $k/2$ vertices, which is in turn bounded by $k^kn^{k/2}$ by Lemma~\ref{lem:tuplesetcount} $B)$.
\end{proof}

Applying Lemma~\ref{lem:reductionSSB} a) and Lemma~\ref{lem:breakaway}, we obtain from \eqref{eq:momenttosum2}:
\begin{equation}
\label{eq:momenttosum3}
\integrala{\sigma_n}{x^k}\ =\ \frac{1}{n^{\frac{k}{2}+1}}\sum_{\pi\in\Ncal\Pcal\Pcal(k)}\sum_{\ubar{t}\in T_{n,k}(\pi)}w(\ubar{t})a_n(\ubar{t})\ +\ R^{(1)}_n \ + \ R^{(2)}_n,
\end{equation}
where $R^{(2)}_n\to 0$ in expectation and in probability under $(E3)$ and also almost surely under $(E3')$.
Applying Lemma~\ref{lem:reductionSSB} b) and Lemma~\ref{lem:breakaway}, we obtain from \eqref{eq:momenttosum3}:
\begin{equation}
\label{eq:momenttosum4}
\integrala{\sigma_n}{x^k}\ =\ \frac{1}{n^{\frac{k}{2}+1}}\sum_{\pi\in\Ncal\Pcal\Pcal(k)}\sum_{\substack{\ubar{t}\in FT_{n,k}(\pi)\\ \dot{\cup} BT_{n,k}(\pi)}}w(\ubar{t})a_n(\ubar{t})\ +\ R^{(1)}_n \ + \ R^{(2)}_n \ + \ R^{(3)}_n,
\end{equation}
where $R^{(3)}_n\to 0$ in expectation and in probability under $(E3)$ and also almost surely under $(E3')$. Finally, applying Lemma~\ref{lem:reductionSSB} c) and Lemma~\ref{lem:breakaway}, we obtain from \eqref{eq:momenttosum4}:
\begin{equation}
\label{eq:momenttosum5}
\integrala{\sigma_n}{x^k}\ =\ \frac{1}{n^{\frac{k}{2}+1}}\sum_{\pi\in\Ncal\Pcal\Pcal(k)}\sum_{\ubar{t}\in BT_{n,k}(\pi)}w(\ubar{t})a_n(\ubar{t})\ +\ R^{(1)}_n \ + \ R^{(2)}_n \ + \ R^{(3)}_n\ +\ R^{(4)}_n,
\end{equation}
where $R^{(4)}_n\to 0$ in expectation and almost surely. In total, setting $R_n\defeq R^{(1)}_n +  R^{(2)}_n  + R^{(3)}_n + R^{(4)}_n$, we obtain from above observations that
\begin{equation}
\label{eq:momenttosum6}
\integrala{\sigma_n}{x^k}\ =\ \frac{1}{n^{\frac{k}{2}+1}}\sum_{\pi\in\Ncal\Pcal\Pcal(k)}\sum_{\ubar{t}\in BT_{n,k}(\pi)}w(\ubar{t})a_n(\ubar{t})\ +\ R_n,
\end{equation}
where $R_n$ is a random variable that converges to zero in expectation and in probability if $(E1)$ and $(E3)$ are assumed and also almost surely if $(E1')$ and $(E3')$ are assumed.

It remains to investigate the first summand on the r.h.s.\ of  \eqref{eq:momenttosum6}. To this end, define for each $\pi\in\Ncal\Pcal\Pcal(k)$:
\begin{equation}
\label{eq:defJwpi}
J_w(\pi)\defeq 	\lim_{n\to\infty} \frac{1}{n^{\frac{k}{2}+1}}\sum_{\ubar{t}\in BT_{n,k}(\pi)}w(\ubar{t}).
\end{equation}
We will show below in Lemma~\ref{lem:Jwpi} that this limit actually exists and how it can be calculated recursively. For now, we take existence for granted. In the next lemma we study the set $BT_{n,k}(\pi)$. We will call elements $\ubar{t}\in BT_{n,k}(\pi)$ $\pi$\emph{--backtracking}. Note that any such $\ubar{t}$ has $k/2$ edges, where each edge is traversed exactly twice. Further, it has $k/2+1$ vertices, so that the graph of $\ubar{t}$ spans a double edged tree. For each $\pi\in\Ncal\Pcal\Pcal(k)$, we denote by $\ubar{t}^{\pi}\in BT_{n,k}(\pi)$ the \emph{canonical} $\pi$\emph{--backtracking} path constructed as follows: Set $t^{\pi}_1=1$ and $t^{\pi}_2=2$, thus determining the edge $e_1$ of $\ubar{t}^{\pi}$. Then if $t^{\pi}_1,\ldots,t^{\pi}_{\ell}$, $2\leq\ell<k$ have been constructed, proceed for $t^{\pi}_{\ell+1}$ as follows: If $e_{\ell}\sim_{\pi}e_{\ell'}$ for some $1\leq\ell'<\ell$, then set $t^{\pi}_{\ell+1}\defeq t^{\pi}_{\ell'}$. Otherwise, set $t^{\pi}_{\ell+1}\defeq \max(t^{\pi}_{1},\ldots,t^{\pi}_{\ell})+1$. 

\begin{lemma}
\label{lem:Jwpiconvergence}
Let $\pi\in\Ncal\Pcal\Pcal(k)$ be arbitrary.
\begin{align*}
a)\quad & BT_{n,k}(\pi) = \{(g(t^{\pi}_1),\ldots,g(t^{\pi}_k))\, | \, g:\oneto{k/2+1}\to\oneto{n} \text{ is injective}\}.\\ 
\quad & \text{In particular: }\forall\, n\geq k:\ \#BT_{n,k}(\pi) = n(n-1)\cdots (n-k/2).\\
b)\quad & \frac{1}{n^{\frac{k}{2}+1}}\sum_{\ubar{t}\in BT_{n,k}(\pi)}w(\ubar{t})\E a_n(\ubar{t})\ \xrightarrow[n\to\infty]{} \ J_w(\pi).\\
c)\quad & \V\left(\frac{1}{n^{\frac{k}{2}+1}}\sum_{\ubar{t}\in BT_{n,k}(\pi)}w(\ubar{t})a_n(\ubar{t})\right)\ \xrightarrow[n\to\infty]{} \ 0.
\end{align*}
and in c) the convergence is summably fast if condition $(E3')$ is assumed instead of $(E3)$.
\end{lemma}
\begin{proof}
Statement a) is clear.\newline
For b), we calculate using condition $(A2)$:
\[
\frac{1}{n^{\frac{k}{2}+1}}\sum_{\ubar{t}\in BT_{n,k}(\pi)}w(\ubar{t})\E a_n(\ubar{t}) = \frac{1}{n^{\frac{k}{2}+1}}\sum_{\ubar{t}\in BT_{n,k}(\pi)}w(\ubar{t})(\E a_n(\ubar{t}) -1) + \frac{1}{n^{\frac{k}{2}+1}}\sum_{\ubar{t}\in BT_{n,k}(\pi)}w(\ubar{t})
\]
Since the second summand on the r.h.s.\ converges to $J_w(\pi)$, it suffices to show that the first summand on the r.h.s.\ converges to zero, which follows from condition $(A2)$ and part a):
\[
\frac{1}{n^{\frac{k}{2}+1}}\sum_{\ubar{t}\in BT_{n,k}(\pi)}\abs{w(\ubar{t})(\E a_n(\ubar{t}) -1)} \leq \frac{1}{n^{\frac{k}{2}+1}}\sum_{\ubar{t}\in BT_{n,k}(\pi)}W^k C^{(k/2)}_n \xrightarrow[n\to\infty]{} 0.
\] 
For c), let $\hat{\pi}$ be a square partition of $\pi$, then if $\hat{\pi}$ is non-traversing and $\ubar{t}$, $\ubar{t}'\in BT_{n,k}(\pi)$ with $(\ubar{t},\ubar{t}')\sim\hat{\pi}$, then
\begin{align*}
\abs{\E a_n(\ubar{t})a_n(\ubar{t}') - \E a_n(\ubar{t}) \E a_n(\ubar{t}')} &\leq \abs{\E a_n(\ubar{t})a_n(\ubar{t}') -1} + \abs{\E a_n(\ubar{t}) -1}\abs{\E a_n(\ubar{t}')} + \abs{\E a_n(\ubar{t}')-1}\\
&\leq C^{(k)}_n + C^{(k/2)}_nC(k) + C^{(k/2)}_n,
\end{align*}
so

\begin{align*}
&\frac{1}{n^{k+2}}\sum_{\substack{\ubar{t},\ubar{t'}\in BT_{n,k}(\pi) \\ (\ubar{t},\ubar{t}')\sim\hat{\pi}}}	\abs{w(\ubar{t})w(\ubar{t}')}\abs{\E a_n(\ubar{t})a_n(\ubar{t}') - \E a_n(\ubar{t}) \E a_n(\ubar{t}')} \\
&\leq \frac{(\# BT_{n,k}(\pi))^2}{n^{k+2}}W^{2k}(C^{(k)}_n + C^{(k/2)}_nC(k) + C^{(k/2)}_n)
\end{align*}
which converges to zero, and this convergence is summably fast if the sequences $(C^{(\ell)}_n)_n$ converge to zero summably fast for all $\ell$.

Now if $\hat{\pi}$ is traversing, we can merely achieve the bound
\[
\abs{\E a_n(\ubar{t})a_n(\ubar{t}') - \E a_n(\ubar{t}) \E a_n(\ubar{t}')} \leq C(2k) + C(k)^2.
\]
On the other hand, by Lemma~\ref{lem:tuplescount2}, $\# \hat{T}_{n,2k}(\hat{\pi}) \leq n^{k-1}o(n^2)B^{k-2}$ where $o(n^2)$ can be replaced by $O(n^{2-\delta})$ in case we assume condition $(E3')$ to hold instead of $(E3)$. Therefore, 
\begin{align*}
&\frac{1}{n^{k+2}}\sum_{\substack{\ubar{t},\ubar{t'}\in BT_{n,k}(\pi) \\ (\ubar{t},\ubar{t}')\sim\hat{\pi}}}	\abs{w(\ubar{t})w(\ubar{t}')}\abs{\E a_n(\ubar{t})a_n(\ubar{t}') - \E a_n(\ubar{t}) \E a_n(\ubar{t}')} \\
&\leq \frac{\#\hat{T}_{n,2k}(\hat{\pi})}{n^{k+2}}W^{2k}(C(2k) + C(k)^2)\leq \frac{n^{k-1}o(n^2)}{n^{k+2}}B^{k-2}W^{2k}(C(2k) + C(k)^2)
\end{align*}
which converges to zero, and this convergence is summably fast if condition $(E3)$ is replaced by $(E3')$, since then $o(n^2)$ can be replaced by $O(n^{2-\delta})$.
\end{proof}
\begin{theorem}
\label{thm:momentconvergence}
Let $k\in\N$ be arbitrary, then it holds
\[
\integrala{\sigma_n}{x^k} \xrightarrow[n\to\infty]{} \sum_{\pi\in\Ncal\Pcal\Pcal(k)}J_w(\pi)
\]
in expectation and in probability, and also almost surely if the conditions $(E1)$ and $(E3)$ are replaced by their stronger counterparts $(E1')$ and $(E3')$, and the sequences $(C^{(\ell)}_n)_n$ from condition $(A2)$ are assumed to converge to zero summably fast.
\end{theorem}
\begin{proof}
Starting from \eqref{eq:momenttosum6} and using the definition in \eqref{eq:defJwpi}, the statement follows with Lemma~\ref{lem:Jwpiconvergence} and Lemma~\ref{lem:convergencemodes}.
\end{proof}

\subsubsection{Analysis of the limiting moments.}
\label{sec:limitingmoments}
First, we establish that the limits $J_w(\pi)$ for $\pi\in\Ncal\Pcal\Pcal(k)$ which were defined in \eqref{eq:defJwpi} actually exist. It turns out this is a limit of a Riemann sum (see also \cite{Bogachev:Molchanov:Pastur:1991}), thus a Riemann integral.
We calculate 
\begin{align}
\frac{1}{n^{\frac{k}{2}+1}}\sum_{\ubar{t}\in BT_{n,k}(\pi)}w(\ubar{t}) 
&\ =\ \frac{1}{n^{\frac{k}{2}+1}} \sum_{\ubar{t}\in BT_{n,k}(\pi)}\prod_{\{t_i,t_j\}\in E_{\ubar{t}}} w\left(\frac{\abs{t_i-t_j}}{n}\right)\notag\\
&\ =\ \frac{1}{n^{\frac{k}{2}+1}} \sum_{\substack{v_1,\ldots,v_{k/2+1}=1\\ \text{all distinct}}}^n\prod_{\{r,s\}\in E_{\ubar{t}^{\pi}}}w\left(\frac{\abs{v_r-v_s}}{n}\right)\notag\\	
&\ =\ \frac{1}{n^{\frac{k}{2}+1}} \sum_{v_1,\ldots,v_{k/2+1}=1}^n\prod_{\{r,s\}\in E_{\ubar{t}^{\pi}}}w\left(\frac{\abs{v_r-v_s}}{n}\right)\ +\ o(1)\notag\\
&\ =\ \int_{[0,1]^{\frac{k}{2}+1}} \prod_{\{r,s\}\in E_{\ubar{t}^{\pi}}}w(\abs{x_r-x_s})\prod_{u\in V_{\ubar{t}^{\pi}}}\de x_u \ +\ o(1)	 \label{eq:Jwpilimit}
\end{align}
In the first step we used the definition of $w(\ubar{t})$, and by slight abuse of notation we identify the abstract edge set $E_{\ubar{t}}=\{e_1,\ldots,e_k\}$ with the family $(\phi_{\ubar{t}}(e_i))_{i\in\oneto{k}}$. For the second step we used Lemma~\ref{lem:Jwpiconvergence} $a)$, 
in the third step we used that the difference of the two terms in question is bounded by
\[
\frac{W^{\frac{k}{2}}}{n^{\frac{k}{2}+1}}\cdot\#\left\{\ubar{v}\in\oneto{n}^{\frac{k}{2}+1}\ |\ \#V_{\ubar{v}}\leq k \right\} = \frac{W^{\frac{k}{2}}}{n^{\frac{k}{2}+1}}\cdot \left(n^{\frac{k}{2}+1} - n(n-1)\cdots(n-k/2)\right) \xrightarrow[n\to\infty]{} 0,
\]
and in the fourth step we recognize the term as a Riemann sum and provide the limit. The next lemma summarizes our findings and shows how the Riemann integral may be calculated recursively.

\begin{lemma}
\label{lem:Jwpi}
Let $\pi\in\Ncal\Pcal\Pcal(k)$ be arbitrary. Then the limit in the definition of $J_w(\pi)	$ in \eqref{eq:defJwpi} exists and it holds
\[
J_w(\pi)= \int_{[0,1]^{\frac{k}{2}+1}} \prod_{\{r,s\}\in \tilde{E}_{\ubar{t}^{\pi}}}w^2(\abs{x_r-x_s})\prod_{u\in V_{\ubar{t}^{\pi}}}\de x_u,
\]
where $\tilde{E}_{\ubar{t}^{\pi}}\defeq \{\phi_{\ubar{t}^{\pi}}(e_i)\,|\,i\in\oneto{k}\}$ as a set.
In particular, 
\[
J_w(\{\{1,2\}\}) = \int_0^1\int_0^1 w^{2}(\abs{x-y})\de x\de y.
\]
Further, for any block of the form $\{m,m+1\}\in\pi$, it holds with $\tilde{\pi}\defeq \pi\backslash\{m,m+1\}$ (which denotes the partition on $\oneto{k-2}$ after eliminating the block $\{m,m+1\}$ from $\pi$ and relabeling the elements $\{m+2,\ldots,k\}$ according to $a\mapsto a-2$),
\[
J_{w}(\pi) = \int_0^1\varphi(x_{t^{\pi}_m}) J_w(\tilde{\pi}|x_{t^{\pi}_m})\de x_{t^{\pi}_m},
\]
where 
\begin{align*}
\varphi(x)&=\int_0^1 w^2(\abs{x-y}) \de y \\
J_{w}(\pi|x_{t^{\pi}_m}) &= \int_{[0,1]^{\frac{k}{2}}} \prod_{\{r,s\}\in \tilde{E}_{\ubar{t}^{\pi}}} w^2(\abs{x_r-x_s}) \prod_{u\in V_{\ubar{t}^{\pi}}\backslash\{t^{\pi}_m\}}\de x_u
\end{align*}
In particular, if $\varphi\equiv c $ is constant for some $c\in\R$, then $J_w(\pi)=c^{k/2}$, and the following upper bound is always valid:
\[
\forall\,\pi\in\Ncal\Pcal\Pcal(k):\ \abs{J_w(\pi)} \leq W^k.
\]
\end{lemma}
\begin{proof}
The integral equation for $J_w(\pi)$ has been derived in the calculation of \eqref{eq:Jwpilimit}. For the recursion, note that $\{m,m+1\}\in\pi$ implies that $t^{\pi}_m=t^{\pi}_{m+2}$ and $t^{\pi}_{m+1}$ is unique in the tuple $\ubar{t}^{\pi}$. Further, $t^{\pi}_m=t_m^{\tilde{\pi}}$, where $\tilde{\pi}\defeq\pi\backslash\{m,m+1\}$. Therefore, 
\begin{align*}
J_w(\pi) &= \int_{[0,1]^{\frac{k}{2}+1}} \prod_{\{r,s\}\in E_{\ubar{t}^{\pi}}}w(\abs{x_r-x_s})\prod_{u\in V_{\ubar{t}^{\pi}}}\de x_u\\
&= \int_{[0,1]^{\frac{k}{2}}}\int_{[0,1]}w(\abs{x_{t^{\pi}_m}-x_{t^{\pi}_{m+1}}})w(\abs{x_{t^{\pi}_{m+1}}-x_{t^{\pi}_{m+2}}})\de x_{t^{\pi}_{m+1}}\ldots \\
&\qquad\qquad\ldots\prod_{\ell\neq m,m+1} w(\abs{x_{t^{\pi}_{\ell}}-x_{t^{\pi}_{\ell+1}}})\prod_{u\in V_{\ubar{t}^{\pi}}\backslash\{t^{\pi}_{m+1}\}}\de x_{u}\\
&= \int_{[0,1]} \varphi(x_{t^{\pi}_m})\int_{[0,1]^{\frac{k}{2}-1}}\prod_{\ell\neq m,m+1} w(\abs{x_{t^{\pi}_{\ell}}-x_{t^{\pi}_{\ell+1}}})\left(\prod_{\substack{u\in V_{\ubar{t}^{\pi}}\\ u\neq t^{\pi}_m,t^{\pi}_{m+1}}}\de x_{u}\right)\de x_{t^{\pi}_m}\\
&= \int_0^1 \varphi(x_{t^{\pi}_m}) J_{w}(\tilde{\pi}|x_{t^{\pi}_m})\de x_{t^{\pi}_m}.
\end{align*}
\end{proof}
 
 From the findings of this section and Theorem~\ref{thm:momentconvergence}, we conclude
 \begin{corollary}
 \label{cor:weakconvergence}
 In the setting of Theorem~\ref{thm:weightedmatrix}, the ESDs $(\sigma_n)_n$ converge weakly in probability to a deterministic probability measure $\mu$ on $\R$ with moments
 \[
 \forall\,k\in\N:\ \integrala{\mu}{x^k} = \sum_{\pi\in\Ncal\Pcal\Pcal(k)}J_w(\pi).
 \]
Further, $\sigma_n\to\mu$ weakly almost surely if conditions $(E1)$ and $(E3)$ are strengthened to $(E1')$ and $(E3')$ respectively. The weak limit $\mu$ has compact support and vanishing odd moments. In particular, it is symmetric and uniquely determined by its moments. 
 \end{corollary}
 \begin{proof}
 The limiting moments in Theorem~\ref{thm:momentconvergence} are bounded by $W^k\#\Ncal\Pcal\Pcal(k)\one_{2\N}(k)\leq (4W)^k\one_{2\N}(k)$, where we used  Lemma~\ref{lem:Jwpi} and a well-known bound on the Catalan-numbers (e.g.\ \cite{Anderson:Guionnet:Zeitouni:2010}). Therefore, the moments satisfy the Carleman condition, thus admit at most one probability measure. By the method of moments for random probability measures (Theorem 3.5 in \cite{Fleermann:2019}), the weak convergence statements in Corollary~\ref{cor:weakconvergence} follow from the stochastic moment convergence in Theorem~\ref{thm:momentconvergence}.
 Also, by the bound on the limiting moments which we identified in the beginning of the proof, and by Lemma 3.13 in \cite{Nica:Speicher:2006}, $\mu$ has compact support. Therefore, the vanishing odd moments allow to conclude that $\mu$ is symmetric (e.g. \cite[p.134]{Tao:2012}).
 \end{proof}

\subsubsection{Characterization of the semicircle law.}
\label{sec:characterization}
We have seen that the ESDs of the ensemble $(X_n)_n$ as in Theorem~\ref{thm:weightedmatrix} converge weakly to the unique symmetric probability distribution $\mu$ on $(\R,\Bcal)$ with moments

\[
\forall\, k\in\N: \integrala{\mu}{x^k} = \sum_{\pi\in\Ncal\Pcal\Pcal(k)} J_{w}(\pi).
\]
and limiting variance 
\[
\varphi_0 \defeq  \int_0^1\varphi(x)\de x,
\]
where
\[
\forall\, x\in [0,1]:~\varphi(x)=\int_0^1 w^2(\abs{x-y}) \de y.
\]

It follows that the ESDs of $(\varphi_0^{-1/2}X_n)_n$ have limiting variance $1$, and the question now is when the semicircle law holds:

\begin{lemma}
Let $(X_n)_n$ be an ensemble as in Theorem~\ref{thm:weightedmatrix}. 
\begin{enumerate}
\item The asymptotic variance $\varphi_0$ can be calculated by
\[
\varphi_0 = 2\int_0^1 (1-x)w^2(x)\de x.
\]
In particular, the following statements are equivalent:
\begin{enumerate}[a)]
\item $\varphi_0 =0$.
\item $w = 0$ on $[0,1]$ $\lebesgue$-almost surely.	
\end{enumerate}

\item Assume that $\varphi_0>0$. Then the following statements are equivalent:
\begin{enumerate}[a)]
\item The semicircle law holds for $\frac{1}{\sqrt{\varphi_0}}X_n$.
\item $\varphi$ is constant, in particular, $\varphi\equiv\varphi_0$.
\item $\varphi \equiv \int_0^1 w^2(x)\de x$,
\item $w^2$ is $\lebesgue$-a.s. symmetric around $1/2$, i.e.\ $w^2(x)=w^2(1-x)$ for $\lebesgue$-a.a. $x\in[0,1]$.
\end{enumerate}
\end{enumerate}

\end{lemma}
\begin{proof}
We prove $(2)$ first:
Denote by $\tilde{\mu}$ the limiting spectral distribution of the ESDs of $\frac{1}{\sqrt{\varphi_0}}X_n$. Then
\begin{equation}
\label{eq:mutilde}
\forall\,k\in\N: \integrala{\tilde{\mu}}{x^k} =\frac{1}{\varphi_0^{k/2}}\sum_{\pi\in\Ncal\Pcal\Pcal(k)} J_{w}(\pi).
\end{equation}
\noindent\underline{$b)\Rightarrow a)$} This follows immediately with \eqref{eq:mutilde} and the last statement in Lemma~\ref{lem:Jwpi}. \newline
\underline{$a)\Rightarrow b)$}: Assume that $\varphi$ is not $\lebesgue$-a.s.\ constant on $[0,1]$. Denote by $\tilde{\mu}$ the LSD of $\varphi_0^{-1/2}X_n$. Then
\[
\integrala{\tilde{\mu}}{x^4} = \frac{1}{\varphi_0^2}\integrala{\mu}{x^2}= \frac{1}{\varphi_0^2}\sum_{\pi\in\Ncal\Pcal\Pcal(k)} J_{w}(\pi).
\]
Since $\Ncal\Pcal\Pcal(k) = \{\pi_1,\pi_2\}$, where $\pi_1=\{\{1,2\},\{3,4\}\}$ and $\pi_2=\{\{1,4\},\{2,3\}\}$ and $\ubar{t}^{\pi_1}=(1,2,1,3)$, $\ubar{t}^{\pi_2}=(1,2,3,2)$, we obtain
\[
J_w(\pi_1)=\int_{[0,1]^3}w^2(\abs{x-y})w^2(\abs{x-z})\de x \de y \de z = \int_0^1\varphi^2(x)\de x.
\]
and
\[
J_w(\pi_2)=\int_{[0,1]^3}w^2(\abs{x-y})w^2(\abs{y-z})\de x \de y \de z = \int_0^1\varphi^2(y)\de y.
\]
As a result,
\[
\integrala{\tilde{\mu}}{x^4} = \frac{2}{\varphi_0^2}\int_0^1 \varphi^2(x)\de x > \frac{2}{\varphi_0^2} \bigabs{\int_0^1\varphi(x)\de x}^2 = 2 = \#\Ncal\Pcal\Pcal(4) = \integrala{\sigma}{x^4}
\]
where the second step follows from Jensen's strict inequality, since $\varphi$ is not constant $\lebesgue$-a.s.\newline
\underline{$b)\Leftrightarrow d)\Rightarrow c)$}:  We calculate
\begin{align}
\varphi(x)&=\int_0^x w^2(x-y)\de y + \int_x^1 w^2(y-x)\de y\notag\\
&= \int_0^x w^2(z)\de z + \int_0^{1-x}w^2(z)\de z\\
&= \int_0^1 w^2(z)\de z - \int_x^1 w^2(z)\de z + \int_x^1 w^2(1-z)\de z\notag\\
&= \int_0^1 w^2(z)\de z + \int_x^1\left[w^2(1-z) - w^2(z)\right] \de z \label{eq:phiequation}
\end{align}
Now let $v(z)\defeq w^2(1-z) - w^2(z)$ for all $v\in[0,1]$. Since $v$ is Riemann integrable on $[0,1]$, it is continuous on a set $C\subseteq[0,1]$ with $\lebesgue(C)=1$. But then with \eqref{eq:phiequation}, $\varphi$ is differentiable at every point $z\in C$ with derivative $w^2(z)-w^2(1-z)$. \newline
Now if $b)$ holds, that is, $\varphi$ is constant, then 
\[
\forall\, z\in C: 0=\varphi'(z)=w^2(z)-w^2(1-z),
\]
which shows statement $d)$. On the other hand, if $d)$ holds, then the second integral in \eqref{eq:phiequation} vanishes, so $\varphi$ is constant with $\varphi(x) = \int_0^1w^2(z)\de z$ for all $x\in[0,1]$, which shows $c)$.\newline
\underline{$c)\Rightarrow a)$} This is immediate.
For $(1)$, we start with equation \eqref{eq:phiequation} and obtain
\begin{align*}
&\varphi_0 = \int_0^1\varphi(x) \de x \\
&=\int_0^1 \int_0^1 w^2(z)\de z \de x + \int_0^1\int_x^1\left[w^2(1-z) - w^2(z)\right] \de z \de x\\
&= \int_0^1 w^2(z)\de z + \int_0^1\int_0^z\left[w^2(1-z) - w^2(z)\right]\de x \de z\\
&= \int_0^1(1-z)w^2(z)\de z + \int_0^1 z w^2(1-z)\de z = 2 \int_0^1(1-z)w^2(z)\de z
\end{align*}
and the last integral is zero iff $w^2(z)=0$ $\lebesgue$-a.s.\ iff $w(z)=0$ $\lebesgue$-a.s.
\end{proof}

\appendix
\section{Proof of Lemma~\ref{lem:AUGaussian}}
\label{sec:ProofAUGaussian}
\begin{proof}
Let $\ell,n\in\N$ and $P_1,\ldots,P_{\ell}\in\oneto{n}^2$ be fundamentally different, and let $\delta_1,\ldots,\delta_{\ell}\in\N$ be arbitrary. Then we calculate (with explanations below)
\begin{align*}
&\E a_n(P_1)^{\delta_1}\cdots a_n(P_{\ell})^{\delta_{\ell}} \ =\ \E \prod_{j=1}^{\delta_1+\ldots+\delta_{\ell}} a_n(P_{i(j)})\\
&=\ \sum_{\pi\in\Pcal\Pcal(\delta_1+\ldots+\delta_{\ell})}\prod_{\{r,s\}\in\pi}\E a_n(P_{i(r)})a_n(P_{i(s)})\ \substack{\leq \\\abs{\ldots}}\ \frac{\#\Pcal\Pcal(\delta_1+\ldots+\delta_{\ell})}{n^{\frac{1}{2}\#\{i\,|\,\delta_i=1\}}}
\end{align*}
For the first step, we set $i(1)=\ldots = i(\delta_1)=1$, $i(\delta_1+1) =\ldots = i(\delta_1+\delta_2) = 2,$ and so on. In the second step, we apply Isserlis formular, see \cite{Isserlis:1918} or \cite{Nica:Speicher:2006}. The third step holds after taking the absolute value on the l.h.s., and then the inequality follows since for all $\pi\in\Pcal\Pcal(\delta_1+\ldots+\delta_{\ell})$, 
\[
\prod_{\{r,s\}\in\pi} \abs{\E a_n(P_{i(r)})a_n(P_{i(s)})} \ = \ \prod_{\{r,s\}\in\pi} \abs{\Sigma_n(P_{i(r)},P_{i(s)})}\ \leq\ \frac{1}{n^{\frac{1}{2}\#\{i\,|\,\delta_i=1\}}}, 
\]
since in the worst case, all single random variables are paired by the partition $\pi$. This proves \eqref{eq:AU1} with constants $C(\ell)\defeq\#\Pcal\Pcal(\ell)$, and for \eqref{eq:AU2} we calculate
\begin{align*}
&\E a_n(P_1)^2\cdots a_n(P_{\ell})^2 \ =\ \E \prod_{j=1}^{2\ell} a_n(P_{i(j)})\ =\ \sum_{\pi\in\Pcal\Pcal(2\ell)}\prod_{\{r,s\}\in\pi}\E a_n(P_{i(r)})a_n(P_{i(s)})\\
& = \ 1\ +\ \sum_{\pi\in\Pcal\Pcal(2\ell)\backslash\{\pi^*\}}\prod_{\{r,s\}\in\pi}\E a_n(P_{i(r)})a_n(P_{i(s)}),
\end{align*}
where in the first step we set $i(1)=i(2)=1$, $i(3)=i(4)=2$, etc., in the second step we apply Isserlis' formula, and for the third step we write
$\pi^*= \{\{1,2\},\ldots\{2\ell-1,2\ell\}\}$. Since each $\pi\in\Pcal\Pcal(2\ell)\backslash\{\pi^*\}$ has at least two blocks which do not pair the same index, we find
\[
\prod_{\{r,s\}\in\pi}\abs{\E a_n(P_{i(r)})a_n(P_{i(s)})} \leq \prod_{\{r,s\}\in\pi}\abs{\Sigma_n(P_{i(r)},P_{i(s)})} \leq \frac{1}{n^2}. 
\]
Consequently, \eqref{eq:AU2} holds with sequences $C^{(\ell)}_n\defeq \Pcal\Pcal(2\ell)/n^2$, which are all summable over $n$.
\end{proof}

\section{Auxiliary Lemmata}


\begin{lemma}
\label{lem:breakaway}
Let $(I_n)_n$ be a sequence of finite index sets and let for all $n\in\N$, $(Y_n(i))_{i\in I_n}$ be a family of random variables with uniformly bounded absolute moments of all orders. 
\begin{enumerate}[a)]
\item If $(m_n)_n$ is a sequence of positive real numbers with $\#I_n = o(m_n)$, then
\[
\frac{1}{m_n}\sum_{i\in I_n}Y_n(i)\ \xrightarrow[n\to\infty]{} \ 0 \qquad \text{in expectation and in probability}.
\] 
\item If $(m_n)_n$ is a sequence of real numbers so that for some $p\in\N$, $(\#I_n/m_n)^p$ is summable (e.g.\ $m_n=n^{\delta}\#I_n$ for some $\delta>0$), then
\[
\frac{1}{m_n}\sum_{i\in I_n}Y_n(i) \ \xrightarrow[n\to\infty]{}\  0 \qquad \text{in expectation and almost surely}.
\] 
\end{enumerate}
\end{lemma}
\begin{proof}
Let $(D_k)_k$ be positive constants for all $k\in\N$ such that for all $n\in\N$, $i\in I_n$, and $k\in\N$: $\E\abs{Y_n(i)}^k\leq D_k$. Choose $\epsilon>0$ and $k\in\N$ arbitrarily.
Then 
\begin{align*}
\Prob\left(\bigabs{ \frac{1}{m_n}\sum_{i\in I_n}Y_n(i)}>\epsilon \right) &\ \leq\ \left(\frac{\#I_n}{m_n}\right)^k\frac{1}{\epsilon^k}\cdot\frac{1}{(\#I_n)^k}\sum_{\ubar{i}\in I^k_n}\E \abs{Y_n(i_1)\cdots Y_n(i_k)} \\
&\ \leq \ \left(\frac{\#I_n}{m_n}\right)^k\cdot\frac{D_k}{\epsilon^k}.
\end{align*}
For part $a)$ choose $k=1$, and for part $b)$ choose $k=p$.
Convergence in expectation to $0$ is trivial.
\end{proof}

\begin{lemma}
\label{lem:convergencemodes}
Let $z\in\N$ and $(Y_n)_n$ be random variables with $\E\abs{Y_n}^z < \infty$ for all $n\in\N$. If $\E Y_n \to y$ and $\E\abs{Y_n-\E Y_n}^z \to 0$, then $Y_n\to y$ in probability. If in addition, $\E\abs{Y_n-\E Y_n}^z$ is summable, then $Y_n\to y$ almost surely.
\end{lemma}
\begin{proof}
Using Markov's inequality, we calculate for $\epsilon>0$ arbitrary:
\begin{align*}
\Prob(\abs{Y_n-y}>\epsilon)&\ \leq\ \Prob\left(\abs{Y_n -\E Y_n}>\frac{\epsilon}{2}\right)\ +\ \Prob\left(\abs{\E Y_n - y}>\frac{\epsilon}{2}\right)\\ 
&\leq\ \frac{2^z}{\epsilon^z}\E\abs{Y_n-\E Y_n}^z\ +\ \Prob\left(\abs{\E Y_n - y}>\frac{\epsilon}{2}\right).
\end{align*}
The statement follows (also using Borel-Cantelli), since the very last summand vanishes for all $n$ large enough.
\end{proof}

\sloppy
\printbibliography

\vspace{1cm}
\noindent\textsf{(Riccardo Catalano, Michael Fleermann, and Werner Kirsch)\newline
FernUniversit\"at in Hagen\newline
Fakult\"at f\"ur Mathematik und Informatik\newline 
Universit\"atsstra\ss e 1\newline 
58084 Hagen}\newline
\textit{E-mail addresses:}\newline
\texttt{riccardo.catalano@fernuni-hagen.de}\newline
\texttt{michael.fleermann@fernuni-hagen.de}\newline
\texttt{werner.kirsch@fernuni-hagen.de}
\vspace{1cm}

\end{document}